\documentclass[11pt,a4paper,reqno]{amsart}
\usepackage{textcomp}
\usepackage{mathtools}
\mathtoolsset{showonlyrefs}
\usepackage{amssymb}
\usepackage{amsmath,amssymb}
\usepackage{graphicx}
\usepackage[sort&compress]{natbib}
\usepackage[cmtip,all]{xy}

\usepackage[pagebackref=true]{hyperref}
\allowdisplaybreaks
\newtheorem{theorem}{Theorem}[section]
\newtheorem{lemma}[theorem]{Lemma}

\theoremstyle{definition}

\newtheorem{example}[theorem]{Example}

\theoremstyle{remark}
\newtheorem{remark}[theorem]{Remark}

\numberwithin{equation}{section}

\DeclareMathOperator{\sgn}{sgn}
\usepackage{color}

\begin{document}

\title[Quasi-linear differential inequalities on graphs]{Liouville Type Results for Quasilinear Elliptic Inequalities Involving Gradient Terms on Weighted Graphs}


%    author one information
\author[A.T.Duong]{Anh Tuan Duong}
	\address{Anh Tuan Duong, Faculty of Mathematics and Informatics\\ Hanoi University of Science and Technology, 1 Dai Co Viet, Bach Mai, Hanoi, Vietnam}
	\email{tuan.duonganh@hust.edu.vn}
\author[Y.Liu]{Yao Liu}
\address{Yao Liu\\Universität Bielefeld, Fakultät für Mathematik, Postfach 100131, D-33501, Bielefeld, Germany}
\email{yaliu@math.uni-bielefeld.de}
\author[N.C.Minh]{Nguy\^en C\^ong Minh}
	\address{Nguy\^en C\^ong Minh,Faculty of Mathematics and Informatics\\ Hanoi University of Science and Technology, 1 Dai Co Viet, Bach Mai, Hanoi, Vietnam}
	\email{minh.nguyencong@hust.edu.vn}
%\thanks{\noindent Y. Liu was supported by China Scholarship Council (No.202406200019).}
\author[D.T.Quyet ]{Dao Trong Quyet}
\address{Dao Trong Quyet\\Academy of Finance, Dong Ngac, Ha Noi, Viet Nam}
\email{daotrongquyet@hvtc.edu.vn}

\author[Y. Sun]{Yuhua Sun}
\address{Yuhua Sun\\ School of Mathematical Sciences and LPMC, Nankai University, 300071
	Tianjin, P. R. China}
\email{sunyuhua@nankai.edu.cn}
\thanks{\noindent Y. Liu was supported by China Scholarship Council (No.202406200019). Y. Sun was supported by the National Natural Science Foundation of China (No.12371206) and supported by "the Fundamental Research Funds for the Central Universities, No.050-63253083.}

%    \subjclass is required.
\date{}

\dedicatory{}
\maketitle
%    Abstract is required.
\begin{abstract}
	In this paper, we study the following quasi-linear elliptic inequality 
	$$
	\Delta_m u +u^p |\nabla u|^q \leqslant 0
	$$
	on weighted graphs, where $(m,p,q)\in (1,\infty)\times\mathbb{R}\times\mathbb{R}$. According to the ranges of parameters $(m, p, q)$, we establish the non-existence of nontrivial positive solutions under the corresponding sharp volume growth conditions. Our results can be viewed as a discrete generalization of their counterparts on Riemannian manifolds established by \textit{[Sun, Yuhua; Xiao, Jie; Xu, Fanheng, Math. Ann. 384 (2022), no. 3-4, 1309--1341.]}. However, this generalization is far from trivial, many results exhibit significant differences from the manifold setting, highlighting the distinct behaviors and challenges that arise in the discrete weighted graph framework.
\end{abstract}

\section{Introduction}
    In this paper, we are concerned with the existence and non-existence of positive solutions to the following quasi-linear inequality
    \begin{equation}\label{eq1}
	   \Delta _m u+u^p|\nabla u|^q\leqslant 0,
    \end{equation}
   on weighted graph, where $(m,p,q)\in\left( 1,\infty \right) \times \mathbb{R} \times \mathbb{R}$. 
   
Let us first consider the case $q=0$ and $m=2$, then \eqref{eq1} becomes 
   \begin{equation}\label{Equations (*)1}
   	\Delta u+u^p\le 0.
   \end{equation}
   The existence and non-existence of positive solutions of \eqref{Equations (*)1} on weighted graphs were studied in the pioneering paper \cite{GHS23}. When $p\leq 1$, \eqref{Equations (*)1} has no positive solution. When $p>1$, under the so-called $(p_0)$ condition of graphs and the volume growth of a ball
   $$V_o(n)\leq C n^{\frac{2p}{p-1}}\ln ^\frac{1}{p-1}(n)\quad \mbox{ for $n>>1$ },$$
   it was proved that \eqref{Equations (*)1} has no positive solution. One can refer to definition of $V_o(n)$ in (\ref{def-vol}). The sharpness of the volume growth was also discussed in \cite{GHS23}.
   
   The result in \cite{GHS23} has been  generalized in  \cite{GW25}  from \eqref{Equations (*)1} to the quasi-linear elliptic inequality
   \begin{equation}\label{emqt1001}
   	\Delta_{m}u+u^p\le 0,
   \end{equation}
   where the volume growth condition is given by
   $$V_o(n) \leq C  n^{\frac{m p }{p-m+1}}(\ln n)^{\frac{m-1}{p-m +1}}\quad \mbox{ for $n>>1$ }.$$
   This condition has been also proved to be optimal. Note that \eqref{Equations (*)1} is a particular case of \eqref{emqt1001} with $m=2$.
   
   In the case $m=2$ and  $p,q\in \mathbb R$, \eqref{eq1} becomes 
   \begin{equation}\label{Equations (*)2}
   	\Delta u+u^p|\nabla u|^q\le 0.
   \end{equation}
   Very recently, this inequality has been investigated in \cite{HS23}. Under a variety of volume growth conditions depending on the ranges of $(p,q)$, the authors in \cite{HS23} obtained optimal Liouville type theorems for positive solutions of \eqref{Equations (*)2}. More precisely, they divided $\mathbb{R}^2$ into six ranges of $(p,q)$, 
   %$$\begin{aligned} & G1=\{(p, q) \mid p \geq 0,1-p<q<2\}, G2=\{(p, q) \mid q \geq 2\}, \\ & G3=\{(p, q) \mid p<0,1<q<2\}, G4=\{(p, q) \mid p<0, q=1\} ,\\ & G5=\{(p, q) \mid p+q=1, p \geq 0, q>0, \text { or } p+q=1, q<0\}, \\ & G6=\{(p, q) \mid p<1-q, q<1, \text { or }(p, q)=(1,0)\},\end{aligned}$$
   and on each range, the optimal volume growth condition for the non-existence of nontrivial positive solutions was proved.

    In recent years, qualitative properties of partial differential  equations  on  weighted graphs have been considerably studied such as existence and non-existence, boundedness, and blow-up  of solutions, see e.g  
    \cite{Yau15, Gr18, GLY16b, GLY17, GHS23, HS23, HL21, HLW23, HLY20, LY20, Saloff-Coste95, Saloff-Coste97, GW25,LW17,LW18,LY20,Liu23,Liu24,Wu21,Wu24, MPS23,MPS24,MDQ25, PZ26-1, PZ26-2}.  Concerning the quasi-linear operator, some existence results for $m$-Laplace elliptic equation  were obtained in \cite{HS21}. For the  $m$-Yamabe equation and $m$-{K}azdan-{W}arner equation, we refer the readers to the articles \cite{Ge18,ZL19,ZL18, Ge20}. Some qualitative results for parabolic equations involving $m$-Laplace operator were established in  \cite{AT20,HM15} and also  references given there.
   
   Let us mention that \eqref{eq1} on Euclidean setting or on Riemannian manifolds was studied in \cite{SXX22, MP01, Filip09,FP10, CHZ22,MP01,GS14, GSV20}. To the best of our knowledge, there is no work dealing with \eqref{eq1} on weighted graphs. Motivated by \cite{HS23,SXX22,GW25}, in this paper, under some volume growth assumptions and the $(p_0)$ condition,  we propose to  establish non-existence results of nontrivial positive solutions of \eqref{eq1}. We also show the sharpness of the volume growth conditions by constructing concrete examples on the homogeneous trees.
   
   This paper is organized as follows. In Section \ref{s1}, we recall the graph setting and formulate our main results.  Sections \ref{sec-thm1}, \ref{sec Thm2.4} and \ref{sec-thm2} are devoted to the proof of non-existence results. The existence results are presented in the last section.
   
   \section{Graph setting and Main results}\label{s1}

Let $G=(V,E)$ be an infinite, connected, locally finite graph. Here $V$ denotes the vertex set and $E$ is the edge set. We assume that between any two distinct vertices, 
    there is at most one edge, and no vertex 
    is connected to itself. \

    For vertices $x,y\in V$, we write $x\sim y$ if they are connected by an edge. The weight function $\mu : E \rightarrow [0,\infty)$ assigns a positive weight $\mu_{xy} > 0$ to each edge $x \sim y$. Here $\mu$ can also be considered as a map from $V\times V $ to $ [0,\infty)$, where $\mu_{xy}=\mu_{yx}$, and $\mu_{xy}>0$ if and only if $x\sim y$. The triplet $(V,E,\mu)$ is referred to as a  weighted graph, we
    denote it by $(V,\mu)$ for brevity. \

    For a vertex $x\in V$, define vertex measure of $x$ as
    $$
    \mu(x)=\sum\limits_{y\sim x}{\mu _{xy}}.
    $$
    Let $\iota (V)$ be the space of all real-valued functions on $V$.  For $m>1$, the $m$-Laplacian operator $\Delta_m :\iota(V) \rightarrow \iota(V)$ is given by
    \begin{equation}
	\Delta_m u(x)=\frac{1}{\mu (x)} \sum_{y\sim x}\mu_{xy} |u(y)-u(x)|^{m-2}[u(y)-u(x)] ,\quad\mbox{for $u\in \iota (V)$}.
    \end{equation}
    The gradient form $\varGamma$ is defined by
    $$
    \varGamma ( f,g ) =\sum_{y\sim x}{\frac{\mu _{xy}}{2\mu (x)}(f(y)-f(x)) (g(y)-g(x))},\quad\mbox{for $f,g\in \iota (V)$},
    $$
    and the norm of the gradient is given by
    \begin{equation}
	|\nabla u(x)|=\sqrt{\varGamma (u,u)}=\sqrt{\sum_{y\sim x}{\frac{\mu _{xy}}{2\mu (x)}(u(y)-u(x))^2}},\quad \mbox{for $u\in \iota (V)$}.
    \end{equation} 
    See \cite{Gr18} for further details.
    
    In this paper, we say that condition $(p_0)$ holds on $G$: if there exists a constant $p_0>1$ such that for any $x \sim y$,
    \begin{equation}\label{p_0}
	\frac{\mu _{xy}}{\mu \left( x \right)}\geqslant \frac{1}{p_0}. \tag{$p_0$}
    \end{equation}
 
    \
    For any two vertices $x,y\in V$, if there exists a finite sequence of vertices $\{x_k\}_{k=0}^{n}$ such that $x_k \sim x_{k+1}$ 
    for all $k=0,1,...,n-1$, then we call $\{x_k\}_{k=0}^{n}$ a path connecting $x_0$ and $x_n$, and 
    define the length to be $n$. 
    The distance between $x$ and $y$, denoted by $d(x,y)$, is defined as the minimal length of all such paths connecting $x$ and $y$.
    %We define $d(x,y)$ as the minimal length of the path that connects $x$ and $y$. 
    
    For any fixed vertex $o\in V$ and integer $n\geqslant1$, the ball of radius $n$ centered at $o$ is defined by
    $$
    B(o,n)\coloneqq\{x\in V|d(o,x)\leqslant n\}.
    $$
    We recall two notions of volume from Saloff-Coste's paper \cite{Saloff-Coste95}:
    \begin{align}\label{def-vol}
    V_o(n)=\mu (B(o,n))=\sum_{x\in B(o,n)}{\mu (x)} = \sum_{\substack{x\in B(o,n),y\in V}}\mu_{xy}.
    \end{align}
    and
    $$
    W_o(n)=W(B(o,n)) = \sum_{\substack{x \in B(o,n),y\in V \\ d(o,x)<d(o,y)}} \mu_{xy}.
    $$
    The quantity $V_o(n)$ is the standard definition of the volume of balls. Throughout the paper, we impose conditions on $W_o(n)$ rather than on $V_o(n)$. Note that, for any $n\geqslant 1$, we have $W_o(n)\leqslant V_o(n)$,  since $W_o(n)$ only counts those edges that lead away from the fixed center. Thus,  conditions on $W_o(n)$ are more refined than those on $V_o(n)$.

The following example from \cite{Saloff-Coste97} illustrates the difference between $W_o(n)$ and $V_0(n)$ more clearly.
    \begin{example}
        Let us consider the following weighted graph. For each $n\geqslant2$, let $K_n$ denote the complete graph on $n$ vertices, meaning that
        every pair of distinct any two vertices is connected by an edge. Construct a connected graph $G$ by 
        successively identifying a vertex of $K_n$ with a vertex $K_{n+1}$ as illustrated in Figure \ref{fig1}. 
        We endow $G$ with a simple weight function $\mu$, and define $\mu_{xy}=1$ for all adjacent vertices $x\sim y$, and $\mu_{xy}=0$ otherwise.

           For this graph, we have $W_o(0)=1$
   and $ W_o(n)-W_o(n-1)=n+1$. Thus, it follows that $W_0(n)=\frac{(n+1)(n+2)}{2}$. On the other hand, $V_o(n)-V_o(n-1)=(n-2)(n-1)+2n-1=n^2-n+1$. 
   Therefore, $V_o(n)\asymp \frac{1}{3}n^3 \gg W_o(n)$, which indicates $V_o(n)$ grows significantly faster than $W_o(n)$.
    \end{example}

    \begin{figure}[htp]
	\begin{center}
		\includegraphics[scale=0.17]{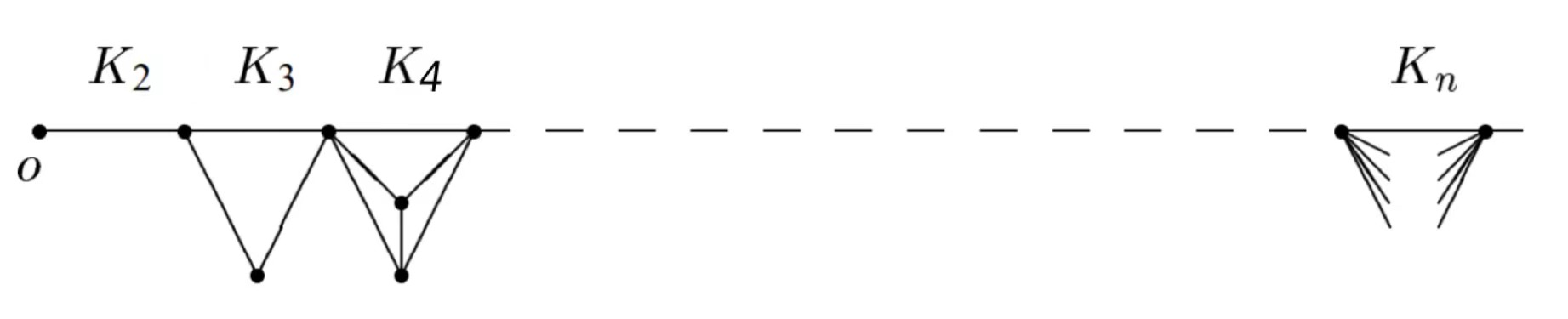}
		\caption{}\label{fig1}
	\end{center}
    \end{figure}

    \begin{remark}
        In what follows we use the following notations: $f\lesssim g$ means that $f\leqslant c g$ for a constant $c>0$. $f\asymp g $ means both $f\lesssim g$ and $g\lesssim f$ hold.
    \end{remark}

    \
    A weighted graph $(V,\mu)$ is said to be $m$-parabolic, if any non-negative function satisfying $\Delta_m u\leq0$ 
    on this graph is constant.  In \cite{Saloff-Coste95}, Saloff-Coste proved that if there exists an infinite sequence
    $r_0<r_1<\cdots<r_i<\cdots$, such that
    $$\sum_{i=0}^{\infty}\left(\frac{(r_{i+1}-r_i)^p}{W_o(r_{i+1}-W_o(r_i))}\right)^{\frac{1}{p-1}}=\infty,$$
    then $(V,\mu)$ is $m$-parabolic.  When $m=2$, the above volume assumption is well-known as Nash-Willimas' test, see \cite{Woess00}.

    \
    Analogous results have also been established in the setting of Riemannian manifolds.
    Let $M$ be a complete, noncompact, connected Riemannian manifold with Riemannian measure $\mu$, and fix a point $o\in M$. It was proved that by 
    Karp \cite{Karp82}, Grigor'yan \cite{Gri85}, Varopoulos \cite{Varo83} independently
    for $m=2$, and by Holopainenen \cite{Hol99} for general $m>1$, if
    $$
    \int^{\infty} \left(\frac{r}{\mu(B(o,r))}\right)^{\frac{1}{m-1}} dr = \infty,
    $$
    then $M$ is $m$-parabolic, i.e., every non-negative function satisfying $\Delta_mu\leq0$ on $M$ is constant.

    \
    To facilitate a systematic analysis, we divide the parameter 
    space  $\mathbb{R}^2$ into six regions $G_1$ through $G_6$ as illustrated in Figure \ref{fig2}.
        \begin{figure}[htp]
    	\begin{center}
    		\includegraphics[scale=0.8]{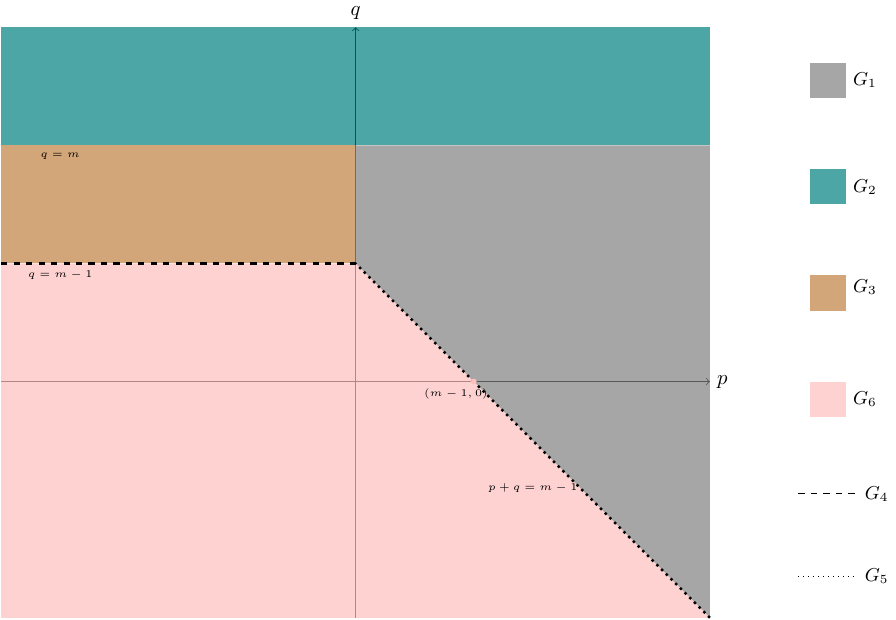}
    		\caption{}\label{fig2}
    	\end{center}
    \end{figure}
    \begin{align*}
	&G_1 = \{(p,q)|p\geqslant0,m-1-p<q<m\}, \quad G_2 = \{(p,q)|q\geqslant m\},\\
	&G_3 = \{(p,q)|p<0,m-1<q<m\}, \quad G_4 = \{(p,q)|p<0,q=m-1\},\\
	& G_5=G_{5.1}\cup G_{5.2}:=\{(p,q)|p+q=m-1,p\geqslant 0,q>0\}\cup \{(p,q)|p+q=m-1,q<0\}, , \\
	&G_6=\{(p,q)|p<m-1-q,q<m-1,or\,(p,q)=(m-1,0)\}.
    \end{align*}

    Next, we present the main theorems of this paper.
        \begin{theorem}\label{Thm1.1}\rm
        Let $G=(V,E,\mu)$ be an infinite, connected, locally finite graph satisfying condition (\ref{p_0}), and fix  $o\in V$.
        Then $\eqref{eq1}$ admits no nontrivial positive solution if any one of the following holds:

        $\mathrm{(I)}$ $(p,q)\in G_1$, and
	  \begin{equation}\label{eq1.4}
		 W_o(n) \lesssim n^{\frac{mp+q}{p+q-m+1}} (\ln n)^{\frac{m-1}{p+q-m+1}}\quad\mbox{for $n\gg1$}.
	  \end{equation}

       $\mathrm{(II)}$  $(p,q)\in G_2$, and
       \begin{equation}\label{eq1.5}
    	 W_o(n)  \lesssim n^m (\ln n)^{m-1}\quad \mbox{for $n\gg1$}.
       \end{equation}

       $\mathrm{(III)}$  $(p,q)\in G_3$, and
       \begin{equation}\label{eq1.6}
    	 W_o(n)  \lesssim n^{\frac{q}{q-m+1}}(\ln n)^{\frac{m-1}{q-m+1}}\quad \mbox{for $n\gg1$}.
       \end{equation}

      $\mathrm{(IV)}$ $(p,q)\in G_4$, and for any  $\alpha>0$,
      \begin{equation}\label{eq1.7}
    	W_o(n)  \lesssim n^{\alpha}\quad \mbox{for $n\gg1$}.
      \end{equation}

      $\mathrm{(V)}$  $(p,q)\in G_{5.1}$ for all $m>1$ or $(p,q)\in G_{5.2}$ with $1<m<3$, and there exists a $k_0>0$ such that   for all $\kappa \in (0,k_0)$,
      \begin{equation}\label{eq1.8}
      	W_o(n)  \lesssim e^{\kappa n},\quad n\gg1.
      \end{equation}
    \end{theorem}

In the case where  $(p,q)\in G_{5.2}$,   with $m\geq 3$, we shall show that \eqref{eq1} has no nontrivial positive solution without any volume growth assumption.

\begin{theorem}\label{Thm1.4}\rm
	 Let $G=(V,E,\mu)$ be an infinite, connected, locally finite graph satisfying condition $(p_0)$, and  $(p,q)\in  G_{5.2}$ with $ m\geq3$, then the inequality $\eqref{eq1}$ admits no nontrivial positive solution.
\end{theorem}

\begin{theorem}\label{Thm1.2}\rm
    Let $G=(V,E,\mu)$ be an infinite, connected, locally finite graph satisfying condition \eqref{p_0}, and  $(p,q)\in G_6$, then the inequality $\eqref{eq1}$ admits no nontrivial positive solution.
\end{theorem}
Recall that when $q=0$ and $p\leq m-1$, \eqref{eq1} has no positive solution, see \cite{GW25}. Henceforth, in Theorem \ref{Thm1.2}, we need only to consider $q\not=0$.

\begin{remark}\rm
Let $G_i^\prime=G_i$, for $i=1,2,3,4$, $G_5^{\prime}=G_5\setminus (m-1,0),$ $G_6^{\prime}=G_6\cup (m-1,0)$. Study the same equation (\ref{eq1}) on  geodesically complete,
connected, non-compact Riemannian manifolds, when $(p,q)\in G_i^{\prime}$ for $i=1,2,3,4$, and $(p, q)\in G_{5.1}$ and $(p,q)\in G_{5.2}$ with $1<m<3$, 
the volume growth on manifolds behaves quite similarly as on weighted graphs.
However, for $(p, q)\in G_{5.2}$ with $m\geq 3$,  and $(p,q)\in G_6^\prime$, 
the Liouville results on graph differ significantly  from the manifold setting, since
 the manifold case admits a sharp exponential volume growth condition for such results,
and no analogous volume condition exists on weighted graphs to guarantee non-existence, see \cite{SXX22}.
\end{remark}

Next, we demonstrate  the sharpness of the first five volume conditions in Theorem $\ref{Thm1.1}$ by constructing the counter-examples on
 homogeneous trees.
Recall that a connected graph $(V,E)$ is called a tree if any two distinct vertices are connected by exactly one path. A homogeneous tree $T_N$ is a tree where
every vertex has degree $N$. Fix an arbitrary vertex $o \in V$ as the root.  Our existence results read as follows.

\begin{theorem}\label{Thm1.3}\rm
    Let $(V,E)=T_N$.

    $\mathrm{(I)}$ For $(p,q)\in G_1$ and any  $\epsilon >0$, there exists a weight $\mu$ on $T_N$ satisfying
    \begin{equation}\label{eq1.9}
		 W_o(n) \asymp n^{\frac{mp+q}{p+q-m+1}} (\ln n)^{\frac{m-1}{p+q-m+1}+\epsilon}\quad\text{ for } n\geqslant2,
	\end{equation}
    such that $\eqref{eq1}$ admits a nontrivial positive solution.

    $\mathrm{(II)}$ For $(p,q)\in G_2$ and any  $\epsilon >0$, there exists a weight $\mu$ on $T_N$ satisfying
    \begin{equation}\label{eq1.10}
    	 W_o(n)  \asymp n^m (\ln n)^{m-1+\epsilon}\quad\text{ for } n\geqslant2,
    \end{equation}
    such that $\eqref{eq1}$ admits a nontrivial positive solution.

    $\mathrm{(III)}$ For $(p,q)\in G_3$ and any  $\epsilon >0$, there exists a weight $\mu$ on $T_N$ satisfying
    \begin{equation}\label{eq1.11}
    	 W_o(n)  \asymp n^{\frac{q}{q-m+1}}(\ln n)^{\frac{m-1}{q-m+1}+\epsilon}\quad\text{ for } n\geqslant2,
    \end{equation}
    such that $\eqref{eq1}$ admits a nontrivial positive solution.

    $\mathrm{(IV)}$ For $(p,q)\in G_4$ and any $\lambda>0$, there exists a weight $\mu$ on $T_N$ satisfying
    \begin{equation}\label{eq1.12}
        W_o(n) \asymp e^{\lambda n}\quad\text{ for } n\geqslant2,
    \end{equation}
    such that $\eqref{eq1}$ admits a nontrivial positive solution.

    $\mathrm{(V)}$  For $(p,q)\in G_{5.1}$ or  $(p,q)\in G_{5.2}$ with $1<m<3$, there exists a weight $\mu$ on $T_N$ and  $\lambda$ satisfying
    \begin{equation}\label{eq1.13}
        W_o(n) \asymp e^{\lambda n}\quad\text{ for } n\geqslant2,
    \end{equation}
    such that $\eqref{eq1}$ admits a nontrivial positive solution.
\end{theorem}

\section{Proof of Theorem \ref{Thm1.1} }\label{sec-thm1}

Before proving Theorem \ref{Thm1.1}, we first establish two key lemmas.

\begin{lemma}\label{lem2.1}\rm
    Let $(V,\mu)$ satisfy condition \eqref{p_0} and let $u$ be a non-negative solution to $\eqref{eq1}$. Then either
    \begin{enumerate}
    \item[1.]{ $u\equiv 0$, or}
    \item[2.]{ $u>0$ elsewhere and satisfies the Harnack-type inequality
    \begin{equation}
		\frac{1}{p_1} \leqslant \frac{u(y)}{u(x)} \leqslant p_1\quad\text{for all } y\sim x,
	\end{equation}
    where  $p_1=1+p_0^{\frac{1}{m-1}}$}.
    \end{enumerate}
\end{lemma}

\begin{proof}\rm
   We first  show that if there exists $x_0\in V$ such that $u(x_0)=0$, then $u\equiv0$. From $\eqref{eq1}$, we have
    \begin{equation*}
	\begin{aligned}
		 \sum_{y\sim x_0}{\frac{\mu_{x_0y}}{\mu (x_0)}|u(y)|^{m-2}u(y)}\leq0.
		\end{aligned}
	\end{equation*}
    Since all terms are non-negative, this implies $u(y)=0$ for all $y\sim x_0$. Connectivity then gives $u\equiv0$.

    Consider now $u>0$. Fix $x\in V$. Using \eqref{eq1}, we have
   \begin{align*}
   	\sum_{y \sim x} \frac{\mu_{xy}}{\mu(x)} |u(y)-u(x)|^{m-2}(u(y)-u(x)) +u(x)^p|\nabla u(x)|^q\le 0.
   \end{align*}
   Since $m>1$ and $u(x)^p|\nabla u(x)|^q\ge 0$, one can see that
   \begin{align}\label{eqb1}
   	\begin{split}
   		\sum_{\substack{y \sim x\\ u(y)\ge u(x)}}\frac{\mu_{xy}}{\mu(x)}(u(y)-u(x))^{m-1}
   		&\le \sum_{\substack{y \sim x\\ u(y)<u(x)}}\frac{\mu_{xy}}{\mu(x)}(u(x)-u(y))^{m-1}\\
   		&\le \sum_{\substack{y \sim x\\ u(y)<u(x)}}\frac{\mu_{xy}}{\mu(x)}u(x)^{m-1}\\
   		&\le u(x)^{m-1}.
   	\end{split}
   \end{align}
  For $y\sim x$ such that $u(y)\ge u(x)$ then \eqref{eqb1} gives
   $$\frac{\mu_{xy}}{\mu(x)}(u(y)-u(x))^{m-1}\le u(x)^{m-1}.$$
   By the condition \eqref{p_0}, we get
   $$u(y)\le \left(1+(\frac{\mu(x)}{\mu_{xy}})^{\frac{1}{m-1}}\right)u(x)\le(1+p_0^{\frac{1}{m-1}})u(x)=p_1u(x),$$
   where $p_1=1+p_0^{\frac{1}{m-1}}$. It is clear that for  $y\sim x$ such that $u(y)< u(x)$ then $u(y)<u(x)<p_1u(x)$. Hence, for $y\sim x$, $u(y)\leq p_1 u(x)$. By exchanging $x$ and $y$, we also have $u(x) \le p_1u(y)$ or $u(y) \ge p_1^{-1}u(x)$. This completes our lemma.
\end{proof}

For brevity, denote
$$
\nabla_{xy}f=f(y)-f(x)\quad \text{for } f\in \iota (V).
$$

Let $\Omega$ be a nonempty subset of $V$. If 
\begin{equation}\label{eq2.5}
    \Delta _m u(x)+u(x)^p|\nabla u(x)|^q\leqslant 0 \quad \mbox{for all $x\in \Omega$},
\end{equation}
then we say that $u$ satisfies \eqref{eq2.5} in $\Omega$.

\begin{lemma}\label{lem2.2}\rm
    Assume $p+q \neq m-1$, $(V,\mu)$ satisfies condition \eqref{p_0}, and $\Omega$ is a nonempty subset of  $V$. Let $u$ be a nontrivial positive function on $V$ which satisfies $\eqref{eq2.5}$, and $\frac{1}{p_1}\leqslant \frac{u(x)}{u(y)} \leqslant p_1$ for any $y\sim x$. Moreover, when $\Omega \neq V$, assume $u$ also satisfies  $u(y)-u(x)\geq 0$ for any $x\in \Omega, y\in \Omega^c$ and $x\sim y$. Then there exists a positive pair $(s,t)$ such that for any $0\leqslant \varphi \leqslant 1$ with compact support in $\Omega$, the following two estimates 
    \begin{equation}\label{eq2.6}
		\begin{split}
			&\sum_{x\in \Omega} \mu(x) u(x)^{p-t} |\nabla u(x)|^q \varphi(x)^s\\
			&\leqslant C_{p_1,t}\left( \frac{(2s)^{\frac{mp+q+t(q-m)}{p+q-m+1}}}{t^{\frac{p(m-1)+t(q-m+1)}{p+q-m+1}}}\right)   
			\left( \sum_{\substack{x,y\in \Omega\\
					\nabla _{xy}\varphi \ne 0}}{\mu _{xy}\varphi \left( x \right) ^su\left( x \right) ^{p-t}|\nabla u\left( x \right)}|^q \right) ^{\small{\frac{m-1-t}{p+q-t}}}
			\\
			&\quad \times \left( \sum_{\substack{x,y\in\Omega \\ \nabla_{xy}\varphi\ne0}} \mu_{xy} |\nabla_{xy} \varphi|^{\frac{mp+q+t(q-m)}{p+q-m+1}}\right) ^{\frac{p+q-m+1}{p+q-t}},
		\end{split}
	\end{equation}
and
    \begin{equation}\label{eq2.7}
	\begin{split}
		\sum_{x\in \Omega}\mu(x) u(x)^{p-t}|\nabla u(x)|^q \varphi(x)^s 
		\lesssim \left( \sum_{\substack{x,y\in\Omega \\ \nabla_{xy}\varphi\ne0}} \mu_{xy} |\nabla_{xy} \varphi|^{\frac{mp+q+t(q-m)}{p+q-m+1}}\right) ^{\frac{p+q-m+1}{p+q-t}}
	\end{split}
    \end{equation}
hold, where $s,t$ satisfy
\begin{equation}\label{eq2.8}
	\left\lbrace \begin{aligned}
		&\frac{mp+q+t(q-m)}{(m-1)p+t(q-m+1)} > 1,\\
		&\frac{p+q-t}{m-1-t} > 1,\\
		&s>\frac{mp+q+t(q-m)}{p+q-m+1}.
    \end{aligned}\right. 
\end{equation}
\end{lemma}

\begin{proof}\rm
    Let  $\varphi \in \iota(V)$ with compact support in $\Omega$. Define $\psi = \varphi^s u^{-t}$, where $(s,t)$ are to be chosen later.

    Multiplying both sides of $\eqref{eq2.5}$ by $\mu (x) \psi (x)$ and summing up over all $x\in \Omega$, we obtain
    $$
	\sum_{x\in\Omega,y\in V}{\mu_{xy} |\nabla_{xy} u|^{m-1} \sgn(\nabla_{xy}u) \psi(x)} + \sum_{x\in\Omega}{\mu(x) u(x)^p |\nabla u(x)|^q \psi(x)} \leqslant 0. 
	$$
    It follows that 
		\begin{equation}\label{eq2.9}
			\begin{aligned}
				&\sum_{x,y\in\Omega}{\mu_{xy} |\nabla_{xy} u|^{m-1} \sgn(\nabla_{xy}u) \psi(x)} \\
				&+ \sum_{x\in\Omega,y\in \Omega^c}{\mu_{xy} |\nabla_{xy} u|^{m-1} \sgn(\nabla_{xy}u) \psi(x)} \\
				& + \sum_{x\in\Omega}{\mu(x) u(x)^p |\nabla u(x)|^q \psi(x)}\leqslant0.
			\end{aligned}
		\end{equation}
    Specially, when $\Omega=V$, we have
    $$
    \sum\limits_{x\in\Omega,y\in \Omega^c}{\mu_{xy} |\nabla_{xy} u|^{m-1} \sgn(\nabla_{xy}u) \psi(x)} = 0.
    $$

    Observe that
    $$
    \sum_{x,y\in\Omega}{\mu_{xy} |\nabla_{xy} u|^{m-1} \sgn(\nabla_{xy}u) \psi(x)} 
    =  - \frac{1}{2} \sum_{x,y\in\Omega}{\mu_{xy} |\nabla_{xy} u|^{m-1} \sgn(\nabla_{xy}u) (\nabla_{xy} \psi)},
    $$
    and
    $$
    \nabla_{xy} \psi = \nabla_{xy}(\varphi^s u^{-t}) 
    = u(y)^{-t} \nabla_{xy}(\varphi^s) + \varphi(x)^s \nabla_{xy}(u^{-t}).
    $$
    Then $\eqref{eq2.9}$ becomes
    \begin{equation}
    	\begin{aligned}\label{eq2.10}
    		&-\frac{1}{2} \sum_{x,y\in\Omega}{\mu_{xy} |\nabla_{xy}u|^{m-1} \sgn(\nabla_{xy}u) \nabla_{xy}(u^{-t}) \varphi(x)^s} + \sum_{x\in\Omega}{\mu(x) u(x)^{p-t} |\nabla u(x)|^q \varphi(x)^s}\\
    		&\qquad \qquad \qquad+\sum_{x\in\Omega,y\in\Omega^c} {\mu_{xy} |\nabla_{xy}u|^{m-1} \sgn(\nabla_{xy}u) \varphi(x)^s u(x)^{-t}} 
    		\\
    		&\qquad \qquad \qquad\leqslant \frac{1}{2} \sum_{x,y\in\Omega}{\mu_{xy} |\nabla_{xy}u|^{m-1} \sgn(\nabla_{xy}u) \nabla_{xy}(\varphi^s) u(y)^{-t}}.
    	\end{aligned}
    \end{equation}

    By the mean value theorem, there exists some $\xi$ taking a value between $u(y)$ and $u(x)$ such that
    $$
    \nabla_{xy}(u^{-t}) = u(y)^{-t} - u(x)^{-t} = -t\xi ^{-t-1} \nabla_{xy}u.
    $$
    Since $\frac{1}{p_1}\leqslant \frac{u(x)}{u(y)} \leqslant p_1$, we have $\frac{u(x)}{p_1} \leqslant \xi \leqslant u(x) p_1$ and
    \begin{equation}
    	\begin{aligned}\label{eq2.11}
    		-\frac{1}{2}& \sum_{x,y\in\Omega}{\mu_{xy} |\nabla_{xy}u|^{m-1} \sgn(\nabla_{xy}u) \nabla_{xy}(u^{-t}) \varphi(x)^s} \\
    		& = \frac{t}{2} \sum_{x,y\in\Omega}{\mu_{xy} |\nabla_{xy}u|^{m-1} \sgn(\nabla_{xy}u) \nabla_{xy}u \varphi(x)^s \xi^{-t-1}} \\
    		&= \frac{t}{2} \sum_{x,y\in\Omega}{\mu_{xy} |\nabla_{xy}u|^{m}   \varphi(x)^s \xi^{-t-1}} \\
    		& \geqslant \frac{t}{2p_1^{t+1}} \sum_{x,y\in\Omega}{\mu_{xy} |\nabla_{xy}u|^{m} u(x)^{-t-1}  \varphi(x)^s} .
    	\end{aligned}
    \end{equation}
    By $0\leqslant \frac{u(y)-u(x)}{u(x)} \leqslant \frac{p_1 u(x)-u(x)}{u(x)} = p_1 - 1$, we obtain
    
    \begin{equation}
    		\begin{aligned}\label{eq2.12}
    		\sum_{x\in\Omega,y\in\Omega^c} &{\mu_{xy} |\nabla_{xy}u|^{m-1} \sgn(\nabla_{xy}u) \varphi(x)^s u(x)^{-t}} \\
    		&\geqslant \sum_{x,y\in \{x\in\Omega,y\in\Omega^c \text{ and } u(x)\neq u(y)\}} {\mu_{xy} |\nabla_{xy} u|^m \frac{u(x)}{\nabla_{xy} u} u(x)^{-t-1} \varphi(x)^s}  \\
    		&\geqslant \frac{1}{p_1-1} \sum_{x\in\Omega,y\in\Omega^c} {\mu_{xy} |\nabla_{xy} u|^m u(x)^{-t-1} \varphi(x)^s }\\
    		& \geqslant \frac{t}{2p_1^{t+1}} \sum_{x\in\Omega,y\in\Omega^c} {\mu_{xy} |\nabla_{xy} u|^m u(x)^{-t-1} \varphi(x)^s } ,
     	   \end{aligned}
    \end{equation}

    where we have used that $2p_1^{t+1}\geqslant t(p_1-1)$ holds for all $t\geqslant0$.

    Combining $\eqref{eq2.11}$ with $\eqref{eq2.12}$, we get
    
    \begin{equation}
    	   \begin{aligned}\label{eq2.13}
    		-\frac{1}{2}  &\sum_{x,y\in\Omega}{\mu_{xy} |\nabla_{xy}u|^{m-1} \sgn(\nabla_{xy}u) \nabla_{xy}(u^{-t}) \varphi(x)^s} \\
    		& + \sum_{x\in\Omega,y\in\Omega^c} {\mu_{xy} |\nabla_{xy}u|^{m-1} \sgn(\nabla_{xy}u) \varphi(x)^s u(x)^{-t}} \\
    		&\geqslant \frac{t}{2p_1^{t+1}}  \sum_{x\in\Omega,y\in\ V} {\mu_{xy} |\nabla_{xy} u|^m u(x)^{-t-1} \varphi(x)^s } \\
    		& = \frac{t}{2p_1^{t+1}} \sum_{x\in \Omega}{( \sum_{y\in V}{\frac{\mu_{xy}}{\mu (x)} \left( |\nabla_{xy}u|^2 \right)^{\frac{m}{2}})  \mu (x) u(x)^{-t-1} \varphi(x)^s} } \\
    		& \geqslant \frac{t}{2p_1^{t+1}} 2^{\frac{m}{2}} \sum_{x\in \Omega} \left( \sum_{y\in V} \frac{\mu_{xy}}{2\mu (x)} |\nabla_{xy} u|^2 \right) ^{\frac{m}{2}} \mu(x) u(x)^{-t-1} \varphi(x)^s \\
    		& \geqslant \frac{t}{p_1^{t+1}} \sum_{x\in \Omega} \mu(x) |\nabla u(x)|^m u(x)^{-t-1} \varphi(x)^s.
    	  \end{aligned}
    \end{equation}

    In particular, when $\Omega=V$, $\eqref{eq2.13}$ follows directly from $\eqref{eq2.11}$.

    By the mean value theorem, there exists some $\eta$ taking a value between $\varphi(y)$ and $\varphi(x)$ such that
    \begin{equation} \label{eq2.14}
	\begin{aligned}
		\nabla_{xy}(\varphi^s) = s \eta^{s-1} (\varphi(y)-\varphi(x)) = s\eta^{s-1} \nabla_{xy}\varphi .
	\end{aligned}
    \end{equation}
    Substituting $\eqref{eq2.14}$ and $\eqref{eq2.13}$ into $\eqref{eq2.10}$, we have
    \begin{equation}\label{eq2.15}
	\begin{aligned}
		\frac{t}{p_1^{t+1}} &\sum_{x\in \Omega} \mu(x) \varphi(x)^s u(x)^{-t-1} |\nabla u(x)|^m  
		+  \sum_{x\in\Omega}{\mu(x) u(x)^{p-t} |\nabla u(x)|^q \varphi(x)^s}  \\
		& \leqslant \frac{s}{2} \sum_{x,y\in \Omega} \mu_{xy} u(y)^{-t} \eta^{s-1} |\nabla_{xy}u|^{m-1} \sgn(\nabla_{xy}u) \nabla_{xy}\varphi .
	\end{aligned}
    \end{equation}
    Observe that 
    $$
    |\nabla u(x)|^2 = \sum_{y\in V} \frac{\mu_{xy}}{2\mu(x)} (\nabla_{xy} u)^2,
    $$
    and $\frac{1}{2p_1} \leqslant \frac{\mu_{xy}}{2\mu(x)} \leqslant \frac{1}{2}$, we derive
    \begin{equation}
	|\nabla_{xy}u| \leqslant \sqrt{2p_1} |\nabla u(x)|, \text{ for any  } y\sim x.
    \end{equation}
    Since $\eta^{s-1} \leqslant \varphi(x)^{s-1} + \varphi(y)^{s-1},\frac{u(x)}{p_1} \leqslant \xi \leqslant u(x)p_1$, we have
    \begin{equation}
    	\begin{aligned}\label{eq2.17}
    		\frac{s}{2}& \sum_{x,y\in \Omega} \mu_{xy} u(y)^{-t} \eta^{s-1} |\nabla_{xy}u|^{m-1} \sgn(\nabla_{xy}u) \nabla_{xy}\varphi   \\
    		&\leqslant \frac{s}{2} \sum_{x,y\in \Omega} \mu_{xy} u(y)^{-t} \left( \varphi(x)^{s-1} + \varphi(y)^{s-1}\right) |\nabla_{xy}u|^{m-1} |\nabla_{xy}\varphi|\\
    		&\leqslant \frac{s}{2} (1+p_1^t) \sqrt{2p_1} \sum_{x,y\in \Omega} \mu_{xy} u(y)^{-t} \varphi(x)^{s-1} |\nabla u(x)|^{m-1} |\nabla_{xy}\varphi| .
    	\end{aligned}
    \end{equation}

    Let
    \begin{equation}\label{eq2.18}
	  a=\frac{mp+q+t(q-m)}{p+q-t}, \quad b= 
       \frac{mp+q+t(q-m)}{(m-1)p+t(q-m+1)},
    \end{equation}
    and $t$ will be chosen later to ensure $a,b\geqslant1$.

    Using Young's inequality, we obtain
    \begin{equation}
    	\begin{aligned}\label{eq2.19}
    		s\sum_{x,y\in \Omega}& \mu_{xy} u(y)^{-t} \varphi(x)^{s-1} |\nabla u(x)|^{m-1} |\nabla_{xy}\varphi| \\
    		&= \sum_{x,y\in \Omega} \left( \mu_{xy}^{\frac{1}{b}} (\frac{t}{2})^{\frac{1}{b}} u(x)^{-\frac{t+1}{b}} |\nabla u(x)|^{\frac{m}{b}} \varphi(x)^{\frac{s}{b}} \right) \\
    		& \quad\times \left( \mu_{xy}^{\frac{1}{a}} s (\frac{t}{2})^{-\frac{1}{b}} u(x)^{-t+\frac{t+1}{b}} |\nabla u(x)|^{m-1-\frac{m}{b}} \varphi(x)^{s-1-\frac{s}{b}} |\nabla_{xy}\varphi|\right)  \\
    		& \leqslant \frac{\epsilon t}{2} \sum_{x,y\in \Omega} \mu_{xy} u(x)^{-t-1} |\nabla u(x)|^m \varphi(x)^s  \\
    		& \quad+ C(\epsilon) t^{1-a} 2^{a-1} s^a \sum_{x,y\in \Omega} \mu_{xy} u(x)^{-t+a-1} |\nabla u(x)|^{m-a} \varphi(x)^{s-a} |\nabla_{xy}\varphi|^a \\
    		& \leqslant \frac{\epsilon t}{2} \sum_{x\in\Omega} \mu(x) u(x)^{-t-1} |\nabla u(x)|^m \varphi(x)^s   \\
    		& \quad+ C(\epsilon) t^{1-a} 2^{a-1} s^a \sum_{x,y\in \Omega} \mu_{xy} u(x)^{-t+a-1} |\nabla u(x)|^{m-a} \varphi(x)^{s-a} |\nabla_{xy}\varphi|^a ,\\
    	\end{aligned}
    \end{equation}
    where $C(\epsilon) = \frac{b^{-\frac{a}{b}}}{a} \epsilon^{-\frac{a}{b}}$.

    Let $\epsilon = \frac{2}{(1+p_1^t)p_1^{t+1}\sqrt{2p_1}}$, and substitute $\eqref{eq2.19}$ into $\eqref{eq2.17}$, we derive that
    \begin{align}\label{eq2.20}
\begin{split}
			\frac{s}{2} &\sum_{x,y\in \Omega} \mu_{xy} u(y)^{-t} \eta^{s-1} |\nabla_{xy} u|^{m-1} \sgn(\nabla_{xy} u) \nabla_{xy}\varphi \\
		&\geqslant \frac{t}{2p_1^{t+1}} \sum_{x\in \Omega} \mu(x) u(x)^{-t-1} |\nabla u(x)|^m \varphi(x)^s  \\
		& + C_{p_1,t} (2s)^a t^{1-a} \sum_{x,y\in \Omega} \mu_{xy} u(x)^{-t+a-1} |\nabla u(x)|^{m-a} \varphi(x)^{s-a} |\nabla_{xy}\varphi|^a,
\end{split}
	\end{align} 
    where 
    \begin{equation*}
	\begin{aligned}
		&C_{p_1,t} = \left( \sqrt{2p_1}\left( 1+{p_1}^t \right) \right) ^{\frac{\left( m-1 \right) p+t\left( q-m+1 \right)}{p+q-t}+1}\left( p_{1}^{t+1} \right) ^{\frac{\left( m-1 \right) p+t\left( q-m+1 \right)}{p+q-t}} C_t,\\
		&C_t = 2^{-1-\frac{\left( m-1 \right) p+t\left( q-m+1 \right)}{p+q-t}}\frac{\left( m-1 \right) p+t\left( q-m+1 \right)}{p+q-t}.
	\end{aligned}
    \end{equation*}
    Combining $\eqref{eq2.20}$ and $\eqref{eq2.15}$, we have
    \begin{equation*}
	\begin{aligned}
		\frac{t}{2p_1^{t+1}} &\sum_{x\in \Omega} \mu(x) \varphi(x)^s u(x)^{-t-1} |\nabla u(x)|^m 
		+  \sum_{x\in \Omega} \mu(x) u(x)^{p-t} |\nabla u(x)|^q \varphi(x)^s \\
		&\leqslant C_{p_1,t} (2s)^a t^{1-a} \sum_{x,y\in \Omega} \mu_{xy} u(x)^{-t+a-1} |\nabla u(x)|^{m-a} \varphi(x)^{s-a} |\nabla_{xy}\varphi|^a .
	\end{aligned}
    \end{equation*}
    It follows that
    \begin{align}\label{eq2.21}
	\begin{split}
			\sum_{x\in \Omega} &\mu(x) u(x)^{p-t} |\nabla u(x)|^q \varphi(x)^s \\
		&\leqslant C_{p_1,t} (2s)^a t^{1-a} \sum_{x,y\in \Omega} \mu_{xy} u(x)^{-t+a-1} |\nabla u(x)|^{m-a} \varphi(x)^{s-a} |\nabla_{xy}\varphi|^a .
	\end{split}
	\end{align}

    Let $\gamma = \frac{p+q-t}{m-1-t},\rho = \frac{p+q-t}{p+q-m+1}$, choose an appropriate $t$ such that $\gamma,\rho >1$, and then apply $\mathrm{H\ddot{o}lder}$'s inequality to RHS of $\eqref{eq2.21}$ to obtain   
	\begin{align}\label{eq2.22}
        \sum_{x,y\in \Omega} &\mu_{xy} u(x)^{-t+a-1} |\nabla u(x)|^{m-a} \varphi(x)^{s-a} |\nabla_{xy}\varphi|^a  \nonumber\\
        & = \sum_{\substack{x,y\in\Omega \\ \nabla_{xy}\varphi\ne0}} \mu_{xy} \left( u(x)^{-t+a-1} |\nabla u(x)|^{m-a} \varphi(x)^{\frac{s}{\gamma}} \right) 
        \left( \varphi(x)^{s-a-\frac{s}{\gamma}} |\nabla_{xy}\varphi|^a\right)  \\
        & \leqslant \left( \sum_{\substack{x,y\in\Omega \nonumber\\ \nabla_{xy}\varphi\ne0}} \mu_{xy} u(x)^{p-t} |\nabla u(x)|^q \varphi(x)^s\right) ^{\frac{1}{\gamma}} 
        \left( \sum_{\substack{x,y\in\Omega \\ \nabla_{xy}\varphi\ne0}} \mu_{xy} \varphi(x)^{s-a\rho} |\nabla_{xy}\varphi|^{a\rho} \right) ^{\frac{1}{\rho}} .
	\end{align}
    Choose sufficiently large $s$ such that $s\geqslant a\rho$. Noting that $0\leqslant\varphi\leqslant1$, we have
    \begin{align}\label{eq2.23}
	\begin{split}
			\sum_{x,y\in \Omega} &\mu_{xy} u(x)^{-t+a-1} |\nabla u(x)|^{m-a} \varphi(x)^{s-a} |\nabla_{xy}\varphi|^a \\
		&\leqslant \left( \sum_{\substack{x,y\in\Omega \\ \nabla_{xy}\varphi\ne0}} \mu_{xy} u(x)^{p-t} |\nabla u(x)|^q \varphi(x)^s\right) ^{\frac{1}{\gamma}} 
		\left( \sum_{\substack{x,y\in\Omega \\ \nabla_{xy}\varphi\ne0}} \mu_{xy}  |\nabla_{xy}\varphi|^{a\rho} \right) ^{\frac{1}{\rho}} .
	\end{split}
    \end{align}
    Substituting $\eqref{eq2.23}$ and $\eqref{eq2.18}$ into $\eqref{eq2.21}$, we get
    \begin{equation*}
	\begin{aligned}
		\sum_{x\in \Omega} &\mu(x) u(x)^{p-t} |\nabla u(x)|^q \varphi(x)^s \\
		&\leqslant C_{p_1,t} (2s)^{\frac{mp+q+t(q-m)}{p+q-t}} t^{-\frac{p(m-1)+t(q-m+1)}{p+q-t}} \left( \sum_{\substack{x,y\in\Omega \\ \nabla_{xy}\varphi\ne0}} \mu(x) u(x)^{p-t} |\nabla u(x)|^q \varphi(x)^s\right)^{\frac{m-1-t}{p+q-t}}  \\
		&\quad \times \left( \sum_{\substack{x,y\in\Omega \\ \nabla_{xy}\varphi\ne0}} \mu_{xy} |\nabla_{xy}\varphi(x)|^{\frac{mp+q+t(q-m)}{p+q-m+1}} \right) ^{\frac{p+q-m+1}{p+q-t}},
	\end{aligned}
    \end{equation*}
    then $\eqref{eq2.6}$ holds.

    Observing that $\sum\limits_{x\in \Omega} \mu(x) u(x)^{p-t} |\nabla u(x)|^q \varphi(x)^s < \infty$ and
    \begin{equation*}
	\begin{aligned}
		\sum_{\substack{x,y\in\Omega \\ \nabla_{xy}\varphi\ne0}} \mu_{xy} u(x)^{p-t} |\nabla u(x)|^q \varphi(x)^s 
		&\leqslant \sum_{x\in\Omega,y\in V} \mu_{xy} u(x)^{p-t} |\nabla u(x)|^q \varphi(x)^s  \\
		&=\sum_{x\in \Omega} \mu(x) u(x)^{p-t} |\nabla u(x)|^q \varphi(x)^s.
	\end{aligned}
    \end{equation*}
    It follows that
    \begin{equation*}
	\begin{split}
		&\left( \sum_{x\in \Omega} \mu(x) u(x)^{p-t} |\nabla u(x)|^q \varphi(x)^s\right) ^{\frac{p+q-m+1}{p+q-t}}   \\
		&\leqslant  C_{p_1,t} (2s)^{\frac{mp+q+t(q-m)}{p+q-t}} t^{-\frac{p(m-1)+t(q-m+1)}{p+q-t}} \left( \sum_{\substack{x,y\in\Omega \\ \nabla_{xy}\varphi\ne0}} \mu_{xy} |\nabla_{xy}\varphi(x)|^{\frac{mp+q+t(q-m)}{p+q-m+1}} \right) ^{\frac{p+q-m+1}{p+q-t}} .
	\end{split}
    \end{equation*}
    Consequently,
    \begin{equation*}
	\begin{aligned}
		\sum_{x\in \Omega} \mu(x) u(x)^{p-t} |\nabla u(x)|^q \varphi(x)^s
		\leqslant &\left( C_{p_1,t}\right) ^{\frac{p+q-t}{p+q-m+1}} (2s)^{\frac{mp+q+t(q-m)}{p+q-m+1}} t^{-\frac{p(m-1)+t(q-m+1)}{p+q-m+1}} \\
		& \times\sum_{\substack{x,y\in\Omega \\ \nabla_{xy}\varphi\ne0}} \mu_{xy} |\nabla_{xy}\varphi|^{\frac{mp+q+t(q-m)}{p+q-m+1}} .
	\end{aligned}
    \end{equation*}
    Then $\eqref{eq2.7}$ follows.
    \end{proof}

\begin{remark}\label{rm2.3}
     In Lemma $\ref{lem2.2}$, since $s$ only needs to be sufficiently large, it is sufficient to verify the existence of $t$. For convenience, let us divide $\mathbb{R}^2\backslash \{p+q=m-1\}$ into four parts:
     \begin{equation*}
		\begin{aligned}
			K_1 = \{(p,q)|p<m-1-q,q\leqslant m-1\},\quad K_2 = \{(p,q)|p\geqslant0,m-1-p<q\leqslant m-1\},  \\
			K_3 = \{(p,q)|p>m-1-q,q>m-1\},\quad K_4 = \{(p,q)|p<0,m-1<q<m-1-p\}.
		\end{aligned}
	\end{equation*}
    
   % \begin{figure}[htp]
	%\begin{center}
	%	\includegraphics[scale=1.2]{figure2.pdf}
	%	\caption{}\label{fig2}
	%\end{center}
   % \end{figure}

    In different cases, we choose $t$ in the following ways
    \begin{enumerate}
    \item[1.]{If $(p,q)\in K_1$, take $t>m-1$;}
    \item[2.]{If $(p,q)\in K_2$, take $0<t<m-1$;}
    \item[3.]{If $(p,q)\in K_3$, take $\max\{\frac{p(1-m)}{q-m+1},0\} < t < m-1$;}
	\item[4.]{If $(p,q)\in K_4$, take $m-1<t<\frac{p(1-m)}{q-m+1}.$}
    \end{enumerate}
\end{remark}

\begin{proof}[Proof of Theorem \ref{Thm1.1} (I)]\rm
    Assume that $u$ is a nontrivial positive solution of inequality $\eqref{eq1}$. By Lemma $\ref{lem2.1}$, it follows that $u$ satisfies $\frac{1}{p_1} \leqslant \frac{u(x)}{u(y)} \leqslant p_1 $ for any $y\sim x$.

    For convenience, let $d(x)=d(o,x)$, $k\in \mathbb N^{+}, B_k = \{x\in V|d(x)\leqslant2^k\}$  and 
    $$
    W(B_k)=\sum\limits_{\substack{x \in B_k,y\in V \\ d(x)<d(y)}} \mu_{xy}.
    $$ 
    Fix integer $n$, define $h_n (x)$ on $V$ as
    \begin{equation}\label{eq2.24}
	h_n(x) = \left\lbrace \begin{aligned}
		 &1,    &d(x)\leqslant n, \\
		 &2-\frac{d(x)}{n}, &\quad n<d(x)<2n, \\
		 &0, &d(x)\geqslant 2n .
	\end{aligned}\right. 
    \end{equation}
    Let
    \begin{equation}\label{eq2.25}
	\varphi_i = \frac{1}{i} \sum_{k=i-1}^{2i-2} h_{2^k}(x).
    \end{equation}
    It follows from the definition that $0\leqslant \varphi_i \leqslant1$, and $\varphi_i = 1$ in $B_{i-1}$, $\varphi_i=0$ in $B_{i-1}^c$. Additionally, for any $x\in B_k - B_{k-1}$, if $k\leqslant i-2$ or $k\geqslant 2i+1$, then $\nabla_{xy}\varphi_i = 0$ for any $y\sim x$; while if $i-1\leqslant k \leqslant 2i$, we have
    $$
    |\nabla_{xy}\varphi_i| \lesssim \frac{1}{i\cdot 2^k},\quad \text{for any }y \sim x.
    $$

    Letting $\Omega=V$ and substituting $\varphi=\varphi_i$ into $\eqref{eq2.7}$,  we have
    \begin{align}
		&\sum_{x\in V} \mu(x) u(x)^{p-t} |\nabla u(x)|^q \varphi_i (x)^s \nonumber  \\
		& \lesssim (2s)^{\frac{mp+q+t(q-m)}{p+q-m+1}} t^{-\frac{p(m-1)+t(q-m+1)}{p+q-m+1}} \sum_{x,y\in V} \mu_{xy} |\nabla_{xy}\varphi_i|^{\frac{mp+q+t(q-m)}{p+q-m+1}}  \nonumber \\
		& \lesssim (2s)^{\frac{mp+q+t(q-m)}{p+q-m+1}} t^{-\frac{p(m-1)+t(q-m+1)}{p+q-m+1}} \sum_{k=i-1}^{2i} \sum_{x\in B_k - B_{k-1}} \sum_{y\sim x} \mu_{xy} |\nabla_{xy}\varphi_i|^{\frac{mp+q+t(q-m)}{p+q-m+1}}  \nonumber  \\
		& \lesssim (2s)^{\frac{mp+q+t(q-m)}{p+q-m+1}} t^{-\frac{p(m-1)+t(q-m+1)}{p+q-m+1}} \sum_{k=i-1}^{2i} W(B_k) \left( \frac{1}{i\cdot 2^k} \right) ^{\frac{mp+q+t(q-m)}{p+q-m+1}}   \nonumber\\
		& \lesssim (2s)^{\frac{mp+q+t(q-m)}{p+q-m+1}} \frac{t^{-\frac{p(m-1)+t(q-m+1)}{p+q-m+1}}}{i^{\frac{mp+q+t(q-m)}{p+q-m+1}}} \sum_{k=i-1}^{2i} W(B_k) 2^{-k\frac{mp+q+t(q-m)}{p+q-m+1}} .
    \end{align}

    Since $G_1 \subset K_2 \cup K_3,$ and by Remark $\ref{rm2.3}$, we choose
    $$
    t = \frac{1}{i} ,
    $$
    and $s$ to be some large fixed constant. From $\eqref{eq1.4}$, we derive that
    \begin{equation}
    	\begin{aligned}\label{eq2.27}
    		\sum_{x\in V}& \mu(x) u(x)^{p-\frac{1}{i}} |\nabla u(x)|^q \varphi_i (x)^s \\
    		&\lesssim i^{-1-\frac{m-1-\frac{1}{i}}{p+q-m+1}} \sum_{k=i-1}^{2i} 2^{\frac{k(m-q)}{i(p+q-m+1)}}  k^{\frac{m-1}{p+q-m+1}} \\
    		&\lesssim i^{-1+\frac{\frac{1}{i}}{p+q-m+1}} \sum_{k=i-1}^{2i} 2^{\frac{k(m-q)}{i(p+q-m+1)}}  \\
    		&\lesssim i^{\frac{\frac{1}{i}}{p+q-m+1}} .
    	\end{aligned}
    \end{equation}

    Letting $i\rightarrow \infty$ in $\eqref{eq2.27}$, we get
    $$
    \sum_{x\in V} \mu(x) u(x)^p |\nabla u(x)|^q < \infty.
    $$
    Then substituting $\varphi=\varphi_i \text{ and }t=\frac{1}{i}$ into $\eqref{eq2.6}$, we arrive
    $$
    \underset{i\rightarrow\infty}{\lim} \sum_{x\in V} \mu(x) u(x)^{p-\frac{1}{i}} |\nabla u(x)|^q \varphi_i (x)^s = 0,
    $$
    that is
    $$
    \sum_{x\in V} \mu(x) u(x)^p |\nabla u(x)|^q = 0,
    $$
    which contradicts the assumption that $u$ is nontrivial. Thus, Theorem \ref{Thm1.1} (I) is proved. 
\end{proof}

\begin{proof}[Proof of Theorem \ref{Thm1.1} (II)]
    The proof is divided into the following three cases:
    
    (II-1).$\quad (p,q)\in \{p>m-1-q,q\geqslant m\}$;

    (II-2).$\quad (p,q)\in \{p=m-1-q,q\geqslant m\}$;

    (II-3).$\quad (p,q)\in \{p<m-1-q,q\geqslant m\}$.

    In case (II-1), $(p,q)\in K_3$. Let
    $$
    t = m-1-\frac{1}{i},
    $$
    and s be some large fixed constant.

    Let $\Omega=V$ in Lemma \ref{lem2.2}, substitute $\varphi=\varphi_i$ from $\eqref{eq2.25}$ into $\eqref{eq2.7}$, and then use the same technique as in $\eqref{eq2.27}$, we derive
    \begin{equation}
    	\begin{aligned}\label{eq2.28}
    		\sum_{x\in V} &\mu(x) u(x)^{p-t} |\nabla u(x)|^q \varphi_i (x)^s  \\
    		& \lesssim (2s)^{\frac{mp+q+t(q-m)}{p+q-m+1}} \frac{t^{-\frac{p(m-1)+t(q-m+1)}{p+q-m+1}}}{i^{\frac{mp+q+t(q-m)}{p+q-m+1}}} \sum_{k=i-1}^{2i} W(B_k) 2^{-k\frac{mp+q+t(q-m)}{p+q-m+1}} .
    	\end{aligned}
    \end{equation}

    Combining $\eqref{eq1.5}$ and $\eqref{eq2.28}$, noting that $(2s)^{\frac{mp+q+t(q-m)}{p+q-m+1}} t^{-\frac{p(m-1)+t(q-m+1)}{p+q-m+1}}$ is uniformly bounded for $i$, we have
    
    \begin{equation}
    	\begin{aligned}\label{eq2.29}
    		\sum_{x\in V} &\mu(x) u(x)^{p-t} |\nabla u(x)|^q \varphi_i (x)^s \\
    		& \lesssim i^{-\frac{mp+q+t(q-m)}{p+q-m+1}} \sum_{k=i-1}^{2i} W(B_k) 2^{-k\frac{mp+q+t(q-m)}{p+q-m+1}} \\
    		& \lesssim i^{-\frac{mp+q+t(q-m)}{p+q-m+1}} \sum_{k=i-1}^{2i} 2^{k(m-\frac{mp+q+t(q-m)}{p+q-m+1})} k^{m-1} \\
    		& \lesssim i^{m-1-\frac{mp+q+t(q-m)}{p+q-m+1}}  \sum_{k=i-1}^{2i} 2^{\frac{k(q-m)}{i(p+q-m+1)}} .
    	\end{aligned}
    \end{equation}

    Substituting $t=m-1-\frac{1}{i}$ into $\eqref{eq2.29}$, we get
    \begin{equation}
	\begin{aligned}
		\sum_{x\in V} \mu(x) u(x)^{p-m+1+\frac{1}{i}} |\nabla u(x)|^q \varphi_i (x)^s 
		\lesssim i^{\frac{q-m}{i(p+q-m+1)}}.
	\end{aligned}
    \end{equation}
    Letting $i\rightarrow \infty$, we also arrive at
    $$
		\sum_{x\in V} \mu(x) u(x)^{p-m+1} |\nabla u(x)|^q  <\infty.
    $$
    Substituting $\varphi=\varphi_i$ and $t=m-1-\frac{1}{i}$ into $\eqref{eq2.6}$, and repeating the same procedures as in the proof of (I), we get
    $$
    \sum_{x\in V} \mu(x) u(x)^{p-m+1} |\nabla u(x)|^q =0,
    $$
    which yields a contradiction with the assumption of $u$.

    In case (II-2), denote $\Omega_k=\{x\in V|0<u(x)<k\}$. Since $u$ is a nontrivial positive solution, there exists a constant $k_0$ such that $\mu(\Omega_k)>0$
    for all $ k>k_0$. Now fix such $k>k_0$, and let $v=\frac{u}{k}$. Then $v$ satisfies
    $$
    \Delta_m v(x) + v(x)^p|\nabla v(x)|^q \leqslant0,\quad x\in V.
    $$
    Obviously $0<v<1$ in $\Omega_k$, we have
    \begin{equation}\label{eq2.31}
  	\Delta_m v(x) + v(x)^{p\prime}|\nabla v(x)|^q 
        \leqslant0,\quad x\in\Omega_k,
    \end{equation}
    where $p^\prime=p+\epsilon$ and $\epsilon>0$, hence $(p^\prime,q)\in \{p^\prime+q=p+q+\epsilon>m-1,q\geqslant m\}$, that is $(p^\prime,q)\in$(II-1).

    According to the definition of $v(x)$ and $\Omega_k$, it is known that $\frac{1}{p_1}\leqslant \frac{v(x)}{v(y)}=\frac{u(x)}{u(y)} \leqslant p_1$ and $v(y)-v(x)\geqslant0$ when $x\in \Omega_k,y\in \Omega_k^c$. Thus, by Lemma$\ref{lem2.2}$ and repeating the same procedure as in case (II-1) except replacing $V$ with $\Omega_k$, we obtain
    \begin{equation}\label{eq2.32}
	  \sum_{x\in\Omega_k} \mu(x) v(x)^{p\prime-m+1} 
        |\nabla v(x)|^q = 0,\quad x\in\Omega_k.
    \end{equation}
    Let $k_i=\max\{u(x)|d(o,x)\leqslant i\}+k_0$, we have $B(o,i)\subset \Omega_{k_i}$. Choosing $k=k_i$ in $\eqref{eq2.32}$, then $v\equiv const.$ in $B(o,i)$, which gives that $u\equiv const.$ in $B(o,i)$. Letting $i\rightarrow\infty$, we obtain that $u\equiv const.$ in $V$, which contradicts the assumption that $u$ is a nontrivial positive solution.

    In case (II-3), let
    $$
    t=m-1+\frac{1}{i},
    $$
    and take the same argument as in case (II-1) to complete the proof of Theorem \ref{Thm1.1} (II).   
\end{proof}

\begin{proof}[Proof of Theorem \ref{Thm1.1} (III)]
    The proof is divided into the following three cases:

    (III-1).$\quad (p,q)\in G_3\cap \{p+q>m-1\}$; 
		
	(III-2).$\quad(p,q)\in G_3\cap \{p+q=m-1\}$; 
		
	(III-1).$\quad(p,q)\in G_3\cap \{p+q<m-1\}$.

    In case (III-1), $(p,q)\in K_3$. Let
    $$
    t = \frac{p(1-m)}{q-m+1}+\frac{1}{i},
    $$
    and $s$ be some large fixed constant.
    Letting $\Omega=V,\varphi=\varphi_i$ in $\eqref{eq2.7}$ and using the same procedure as before, we obtain
    $$
    \sum_{x\in V} \mu(x) u(x)^{p-t} |\nabla u(x)|^q \varphi_i(x)^s 
    \lesssim i^{-\frac{mp+q+t(q-m)}{p+q-m+1}} \sum_{k=i-1}^{2i} W(B_k) 2^{-k\frac{mp+q+t(q-m)}{p+q-m+1}} .
    $$
    Combining with $\eqref{eq1.6}$, we have
   
   \begin{equation}
   	\begin{aligned}\label{eq2.33}
   		\sum_{x\in V} &\mu(x) u(x)^{\frac{pq}{q-m+1}-\frac{1}{i}} |\nabla u(x)|^q \varphi_i(x)^s  \\
   		&\lesssim i^{-\frac{mp+q+t(q-m)}{p+q-m+1}} \sum_{k=i-1}^{2i} 2^{k(\frac{q}{q-m+1}-\frac{mp+q+t(q-m)}{p+q-m+1})} k^{\frac{m-1}{q-m+1}} \\
   		&\lesssim i^{-\frac{mp+q+t(q-m)}{p+q-m+1}} \sum_{k=i-1}^{2i} 2^{\frac{k(m-q)}{i(p+q-m+1)}} k^{\frac{m-1}{q-m+1}} \\
   		&\lesssim i^{\frac{m-1}{q-m+1}-\frac{mp+q+t(q-m)}{p+q-m+1}+1} \\
   		&\lesssim i^{-\frac{q-m}{i(p+q-m+1)}}.
   	\end{aligned}
   \end{equation}
    
    Letting $i\rightarrow\infty$ in $\eqref{eq2.33}$, we have
    $$
    \sum_{x\in V} \mu(x) u(x)^{\frac{pq}{q-m+1}} |\nabla u(x)|^q < \infty.
    $$
    Repeating the proof procedure in Theorem \ref{Thm1.1} (I), we derive
    $$
    \sum_{x\in V} \mu(x) u(x)^{\frac{pq}{q-m+1}} |\nabla u(x)|^q = 0,
    $$
    which contradicts the assumption that $u$ is a nontrivial positive solution.

    In case (III-2), let $ 0<\epsilon<-\frac{p}{2}$ such that $(p^\prime,q)=(p+\epsilon,q)\in $(III-1), then repeat the steps in (II-2) to complete the proof for this case.

    In case (III-3), it suffices to take
    $$
    t = \frac{p(1-m)}{q-m+1}-\frac{1}{i},
    $$
    and the proof as in (III-1).

    Hence, we complete the proof of Theorem \ref{Thm1.1} (III). 
\end{proof}

\begin{proof}[Proof of Theorem \ref{Thm1.1} (IV)]
    In this case, i.e. $p<0, q=m-1$, take
    $$
    t=l(m-1)+\frac{1}{i},\qquad s=-\frac{l(m-1)}{p}+m+\frac{1}{i},
    $$
    where $l>1$ is to be chosen later.

    In Lemma \ref{lem2.2}, substitute $\Omega=V,\varphi=\varphi_i$ into $\eqref{eq2.7}$, combine with \eqref{eq1.7} and repeat the previous arguments to obtain
    \begin{equation}
	\begin{aligned}
		\sum_{x\in V} &\mu(x) u(x)^{p-t} |\nabla u(x)|^q \varphi_i(x)^s \\
		&\lesssim i^{-\frac{mp+q+t(q-m)}{p+q-m+1}} \sum_{k=i-1}^{2i} W(B_k) 2^{-k\frac{mp+q+t(q-m)}{p+q-m+1}}  \\
		&\lesssim i^{-\frac{mp+q-t}{p}} \sum_{k=i-1}^{2i} 2^{k(\alpha-\frac{mp+q-t}{p})} .
	\end{aligned}
    \end{equation}

    Letting $l$ be sufficiently large satisfying that
    $$
    \alpha-\frac{mp+q-t}{p} <0,
    $$
   for all $i$,  we obtain
    \begin{equation}\label{eq2.35}
	\sum_{x\in V} \mu(x) u(x)^{p-t} |\nabla u(x)|^q \varphi_i(x)^s \lesssim i^{1-\frac{mp+q-t}{p}} .
    \end{equation}
    Further, for large enough $l$, the following is also satisfied
    $$
    1-\frac{mp+q-t}{p} < c <0.
    $$
    In $\eqref{eq2.35}$, By letting $i\rightarrow\infty$, we obtain
    $$
    \sum_{x\in V} \mu(x) u(x)^{p-l(m-1)} |\nabla u(x)|^q =0,
    $$
    which contradicts the assumption that $u$ is a nontrivial positive solution. Thus, the Theorem \ref{Thm1.1} (IV) is proved.  
\end{proof}

\begin{proof}[Proof of Theorem \ref{Thm1.1} (V)]
    Divide the proof into two cases:

    %(V-1).$\quad (p,q)\in \{p+ q = m-1,p\geqslant0,q>0\}$; 

    %(V-2).$\quad (p,q)\in \{p+ q = m-1,p> m - 1,q<0\}$.

    From $\eqref{eq1}$, we derive that 
    $$
    \sum_{y\in V} \frac{\mu_{xy}}{\mu(x)} |u(y)-u(x)|^{m-2} (u(y)-u(x)) +u(x)^p |\nabla u(x)|^q \leqslant 0,
    $$
    which means
    \begin{equation}
    	 \begin{aligned}\label{eq2.36}
    		\sum_{y\in V} \frac{\mu_{xy}}{\mu(x)} |u(y)-u(x)|^{m-2} u(y) 
    		& - \sum_{y\in V} \frac{\mu_{xy}}{\mu(x)} |u(y)-u(x)|^{m-2} u(x)\\
    		& + u(x)^p |\nabla u(x)|^q \leq0 .
    	\end{aligned}
    \end{equation}

    Note that when $y\sim x$, 
    \begin{equation}\label{eq2.37}
 	|u(y)-u(x)| \leqslant C_{p_1} u(x),
    \end{equation}
    where $C_{p_1}=\max\{p_1 -1,1-\frac{1}{p_1}\}$.

    Substituting $\eqref{eq2.37}$ into $\eqref{eq2.36}$ and using $p+q=m-1$,  we obtain
    \begin{equation*}
	\begin{aligned}
		\sum_{y\in V} \frac{\mu_{xy}}{\mu(x)} |u(y)-u(x)|^{m-2} u(y) 
		&\leqslant C_{p_1}^{m-2} u(x)^{p+q} - u(x)^p |\nabla u(x)|^q  \\
		&= u(x)^{p+q} (C_{p_1}^{m-2}-u(x)^{-q} |\nabla u(x)|^q ),
	\end{aligned}
    \end{equation*}
    which gives
    \begin{equation}\label{eq2.38}
	 u(x)^{-q} |\nabla u(x)|^q \leqslant C,
    \end{equation}
    where $C = C_{p_1}^{m-2}$.

    In case $(p,q)\in G_{5.1}$, since $p+q=m-1$ and $q>0$, we have
    \begin{equation}\label{eq2.39}
	  |\nabla u(x)| \leqslant C_1 u(x),
    \end{equation}
    where $C_1 = C^{\frac{1}{q}} = C_{p_1}^{\frac{m-2}{q}}$. Plugging this  into $\eqref{eq1}$, we get
    \begin{equation}
 	\Delta_m u(x) +  \frac{1}{C_1} |\nabla 
        u(x)|^{m-1} \leqslant  \Delta_m u(x) + u(x)^p |\nabla u(x)|^q \leqslant 0.
    \end{equation}

    Denote $\Omega_k^{\prime} = \{x\in V|0<u(x)<k,|\nabla u(x)|\ne 0\}$. Since $u$ is a nontrivial positive solution, there exists a constant $k$ such that $\mu(\Omega_k^\prime)>0$. Let $v=\frac{u}{k}$, then $v$ satisfies
    $$
	\Delta_m v(x) + C_2 |\nabla v(x)|^{m-1} \leqslant0,
    $$
    where $C_2=\frac{1}{C_1}=C_{p_1}^{-\frac{m-2}{q}}$. Noting $0<v<1$ on $\Omega_k^\prime$ and by $\eqref{eq2.39}$, we have
    $$
    C_2 |\nabla v(x)| \leqslant v(x) \leqslant1.
    $$
    Consequently,
    \begin{equation}\label{eq2.41}
	  \Delta_m v(x) + C_2 |\nabla v(x)|^{\lambda} 
        \leqslant0,\quad x\in \Omega_k^\prime,
    \end{equation}
    where $\lambda>m-1$ is to be chosen later.

    By the definition of $\Omega_k^\prime$, it is known that when $x\sim y,y\in V\setminus{\Omega_k^\prime},x\in \Omega_k^\prime$,
    \begin{equation*}
	\left\lbrace\begin{aligned}
		 &v(y)>v(x),\text{ when } v(y)\geqslant 1,\\
		&v(y)=v(x),\text{ when } v(y)<1,\text{ and } |\nabla v(y)|=0. 
	\end{aligned}\right.
    \end{equation*}
    In both of the above cases, we arrive at $v(y)-v(x)\geqslant0$.

    For any $x\in \Omega_k^\prime$, we have
    \begin{equation}\label{eq2.42}
	\begin{aligned}
		\Delta_m v(x) 
		&= \sum_{y\in V} \frac{\mu_{xy}}{\mu(x)} |v(y)-v(x)|^{m-1} \sgn(\nabla_{xy}v)   \\
		&= \sum_{y\in \Omega_k^\prime} \frac{\mu_{xy}}{\mu(x)} |v(y)-v(x)|^{m-1} \sgn(\nabla_{xy}v) + \sum_{y\in V\setminus{\Omega_k^\prime}} \frac{\mu_{xy}}{\mu(x)} (\nabla_{xy}v)^{m-1}   \\
		&\geqslant \sum_{y\in \Omega_k^\prime} \frac{\mu_{xy}}{\mu(x)} |v(y)-v(x)|^{m-1} \sgn(\nabla_{xy}v) .
	\end{aligned}
    \end{equation}
    Similarly,
    \begin{equation}
    	\begin{aligned}\label{eq2.43}
    		|\nabla v(x)| 
    		&= \sqrt{\sum_{y\in V} \frac{\mu_{xy}}{2\mu(x)} (\nabla_{xy}v)^2}  \\
    		&= \sqrt{\sum_{y\in \Omega_k^\prime} \frac{\mu_{xy}}{2\mu(x)} (\nabla_{xy}v)^2 + \sum_{y\in (\Omega_k^\prime)^c}\frac{\mu_{xy}}{2\mu(x)} (\nabla_{xy}v)^2}  \\
    		& \geqslant \sqrt{\sum_{y\in \Omega_k^{\prime}} \frac{\mu_{xy}}{2\mu(x)} (\nabla_{xy}v)^2} \eqqcolon |\nabla_{\Omega_k^\prime} v(x)|.
    	\end{aligned}
    \end{equation}
 
   Here we remark  that $|\nabla_{\Omega_k^\prime} v(x)|$ is not the norm of gradient of $u$ in $\Omega_k^{\prime}$.

    Substituting $\eqref{eq2.43}$ and $\eqref{eq2.42}$ into $\eqref{eq2.41}$, we have
    \begin{equation}\label{eq2.44}
	   \sum_{y\in \Omega_k^\prime} \frac{\mu_{xy}} 
          {\mu(x)} |v(y)-v(x)|^{m-1} \sgn(\nabla_{xy}v) + C_2 |\nabla_{\Omega_k^{\prime}} v(x)|^{\lambda}
    	\leqslant 0,\quad x\in \Omega_k^{\prime} .
    \end{equation}
    Multiplying both sides of $\eqref{eq2.44}$ by $\mu(x) h_n^z$, where $h_n^z$ is the same as in $\eqref{eq2.24}$, and then summing  over $x\in\Omega_k^\prime$, we get
    \begin{equation}\label{eq2.45}
	\begin{aligned}
		C_2\sum_{x\in\Omega_k^\prime} \mu(x) |\nabla_{\Omega_k^\prime} v(x)|^{\lambda} h_n(x)^z 
		&\leqslant -\sum_{x,y\in\Omega_k^\prime} \mu_{xy} h_n(x)^z |\nabla_{xy}v|^{m-1} \sgn(\nabla_{xy}v) \\
		&\leqslant \frac{1}{2} \sum_{x,y\in\Omega_k^\prime} \mu_{xy} (\nabla_{xy}h_n^z) |\nabla_{xy}v|^{m-1}   \\
		&= \frac{z}{2} \sum_{x,y\in\Omega_k^\prime} \mu_{xy} \eta^{z-1} (\nabla_{xy}h_n) |\nabla_{xy}v|^{m-1},
	\end{aligned}
    \end{equation}
    where $\eta>0$ takes values between $h_n(x)$ and $h_n(y)$.

    Noting that $|\nabla_{\Omega_k^\prime} v(x)|^2=\sum\limits_{y\in \Omega_k^\prime} \frac{\mu_{xy}}{2\mu(x)} (\nabla_{xy}v)^2$ and $\frac{1}{2p_0}\leqslant \frac{\mu_{xy}}{2\mu(x)} \leqslant\frac{1}{2}$, we have
    \begin{equation}\label{eq2.46}
	|\nabla_{xy}v| \leqslant \sqrt{2p_0} |\nabla_{\Omega_k^\prime} v(x)|,\text{ for any } y\sim x .
    \end{equation}

    Substituting \eqref{eq2.46} and $\eta^{z-1}\leqslant \max(h_n(x)^{z-1},h_n(y)^{z-1}) \leqslant h_n(x)^{z-1}+h_n(y)^{z-1}$ into $\eqref{eq2.45}$, and applying $\mathrm{H\ddot{o}lder}$'s inequality, we obtain
	\begin{align*}
		C_2\sum_{x\in\Omega_k^\prime} &\mu(x) |\nabla_{\Omega_k^\prime} v(x)|^{\lambda} h_n(x)^z \nonumber\\
		& \leqslant \frac{z}{2} \sum_{x,y\in\Omega_k^\prime} \mu_{xy} (h_n(x)^{z-1}+h_n(y)^{z-1}) |\nabla_{xy}h_n| |\nabla_{xy}v|^{m-1} \nonumber\\
		& = z \sum_{x,y\in\Omega_k^\prime} \mu_{xy} h_n(x)^{z-1} |\nabla_{xy}h_n| |\nabla_{xy}v|^{m-1}\nonumber \\
		& \leqslant (2p_0)^{\frac{m-1}{2}}  z \sum_{x,y\in\Omega_k^\prime} \mu_{xy} h_n(x)^{z-1} |\nabla_{xy}h_n| |\nabla_{\Omega_k^\prime} v(x)|^{m-1}   \nonumber\\
		& \leqslant (2p_0)^{\frac{m-1}{2}} z \left( \sum_{x,y\in\Omega_k^\prime} \mu_{xy} |\nabla_{\Omega_k^\prime} v(x)|^{\lambda} h_n(x)^z\right) ^{\frac{m-1}{\lambda}}
		\left( \sum_{x,y\in\Omega_k^\prime} \mu_{xy} |\nabla_{xy}h_n|^z\right) ^{\frac{\lambda-m+1}{\lambda}} ,\nonumber
	\end{align*}
    where we take
    $$
    z=\frac{\lambda}{\lambda-m+1}.
    $$

 The fact that $\sum\limits_{x\in\Omega_k^\prime} \mu(x) |\nabla_{\Omega_k^\prime} v(x)|^{\lambda} h_n(x)^z $ is bounded,  $h_n=1$ in $B(o,n)$, and  $\eqref{eq1.8}$ imply
    
	\begin{align}\label{eq2.48}
		C_2 \sum_{x\in\Omega_k^\prime\cap B(o,n)} \mu(x) |\nabla_{\Omega_k ^\prime} v(x)|^{\lambda} \nonumber&\leqslant \left( (2p_0)^{\frac{m-1}{2}} z\right) ^z \sum_{x,y\in\Omega_k^\prime} \mu_{xy} |\nabla_{xy}h_n|^z \nonumber\\
		&\leqslant \left( \frac{(2p_0)^{\frac{m-1}{2}} z}{n}\right) ^z W(B(o,2n)) \\
		& \lesssim \left( \frac{(2p_0)^{\frac{m-1}{2}} z}{n}\right) ^z e^{2\kappa n} .
	\end{align}

    Here, we have  used $|\nabla_{xy}h_n|\leqslant \frac{1}{n}$ and $|\nabla_{xy}h_n|=0$ for any $x,y\in B(o,n)^c$. Let
    \begin{equation}\label{eq2.49}
	 z=\theta n,
    \end{equation}
    where the positive constant $\theta$ will be determined later. It is easy to see that $\lambda\rightarrow (m-1)_+$ is equivalent to $n\rightarrow \infty$.

    Letting $\lambda\rightarrow(m-1)_+$ in \eqref{eq2.48}, and from \eqref{eq2.49}, we obtain
    \begin{equation}
    	\begin{aligned}\label{eq2.50}
    		\sum_{x\in\Omega_k^\prime} \mu(x) |\nabla_{\Omega_k^\prime}v(x)|^{m-1} 
    		&\leqslant \lim_{\lambda\rightarrow (m-1)_+} \sum_{x\in\Omega_k^\prime} \mu(x) |\nabla_{\Omega_k ^\prime}v(x)|^{\lambda}  \\
    		& \lesssim \lim_{n\rightarrow\infty} \left( \frac{(2p_0)^{\frac{m-1}{2}} z}{n}\right) ^z e^{2\kappa n} \\
    		& = \lim_{n\rightarrow\infty} e^{n(2\kappa+\theta \ln((2p_0)^{\frac{m-1}{2}}\theta))}.
    	\end{aligned}
    \end{equation}

    Since $g(\theta) = \theta \ln((2p_0)^{\frac{m-1}{2}}\theta)$ attains its maximum at $\theta=\frac{1}{(2p_0)^{\frac{m-1}{2}}e}$, if $\kappa$ satisfies 
    $$
      0<\kappa<\kappa_0=\frac{1}{2(2p_0)^{\frac{m-1}{2}}e},
    $$
    there  exists $\theta>0$ such that 
    \begin{equation*}
	  2\kappa+\theta \ln\left( (2p_0)^{\frac{m-1} 
        {2}}\theta\right)  < 0.
    \end{equation*}

    Under the above choice of $\kappa$ and $\theta$, it follow from \eqref{eq2.50} that 
    \begin{equation*}
	  \sum_{x\in\Omega_k^\prime} \mu(x) 
        |\nabla_{\Omega_k^\prime}v(x)|^{m-1} =0.
    \end{equation*}
    Therefore, when $x\in \Omega_k^\prime$, $|\nabla_{\Omega_k^\prime}v(x)|=0$. 
    
    We  next claim that $|\nabla v(x)|=0$ for any $x\in \Omega_k^\prime$ by contradiction argument.  Assume, on the contrary, that  there exists some $x_0\in \Omega_k^\prime$ such that $|\nabla_{\Omega_k^\prime}v(x_0)|=0$ but $|\nabla v(x_0)|\ne 0$. Set $U=\{y|y\sim x_0 \text{ and }u(y)\ne u(x_0)\}$. Clearly, $U\subset (\Omega_k^\prime)^c$. By the definition of $\Omega_k^\prime$, it is known that $v(y)>v(x_o)$ when $y\in(\Omega_k^\prime)^c,x_0\in\Omega_k^\prime,y\sim x_0$ and $v(y)\ne v(x_0)$.

   Noticing  that $u(y)=u(x_0)$ for any $y\in U^c$, we have
    \begin{equation*}
	\begin{aligned}
		\Delta_m v(x_o) 
		&= \sum_{y\in V} \frac{\mu_{x_0y}}{\mu(x_0)} |v(y)-v(x_0)|^{m-1} \sgn(\nabla_{x_0y}v)  \\
		&= \sum_{y\in U} \frac{\mu_{x_0y}}{\mu(x_0)} |v(y)-v(x_0)|^{m-1} \sgn(\nabla_{x_0y}v)  \\
		&\qquad +  \sum_{y\in U^c} \frac{\mu_{x_0y}}{\mu(x_0)} |v(y)-v(x_0)|^{m-1} \sgn(\nabla_{x_0y}v) \\
		&=\sum_{y\in U} \frac{\mu_{x_0y}}{\mu(x_0)} |v(y)-v(x_0)|^{m-1} \sgn(\nabla_{x_0y}v) >0,
	\end{aligned}
    \end{equation*}
    which contradicts with \eqref{eq2.41}.

    In case $(p,q)\in G_{5.2}$ with $1<m<3$, since $p+q=m-1,q<0$, it results from \eqref{eq2.38} that 
    \begin{equation}\label{eq2.51}
	 u(x)\leqslant C_3|\nabla u(x)|,
    \end{equation}
    where $C_3=C^{-\frac{1}{q}}=C_{p_1}^{\frac{2-m}{q}}$.

    Set $\Omega_k = \{x\in V|0<u(x)<k\}$. Since $u$ is a nontrivial positive solution, there exists $k_0>0$ such that $\mu(\Omega_k)>0$ for any $k>k_0$. Fix $k$, let $v=\frac{u}{k}$, then $v$ satisfies
    $$
    \Delta_m v(x)+v(x)^p |\nabla v(x)|^q \leqslant 0,\quad \forall x\in V.
    $$
    It is easy to see that $0<v<1$ in $\Omega_k$. Consequently,
    $$
    \Delta_m v(x)+v(x)^{p+\epsilon} |\nabla v(x)|^q \leqslant 0,\quad \forall x\in \Omega_k,
    $$
    where the constant $\epsilon$ is to be determined.

    Since $(p^{\prime},q)\coloneqq(p+\epsilon,q)\in K_2,p^{\prime}+q=m-1+\epsilon$, we choose $0<t<m-1,s=\frac{mp+q+t(q-m)+m\epsilon}{\epsilon}+1$, and let
    $$
    z=\frac{mp+q+t(q-m)+m\epsilon}{\epsilon}.
    $$
    It is easy to see that $z\rightarrow\infty$ is equivalent to $\epsilon\rightarrow0$.

    Note that $\frac{1}{p_1}\leqslant\frac{v(x)}{v(y)}\leqslant p_1$ and $v(y)-v(x)\geqslant0$ when $x\in \Omega_k$, $y\in \Omega_k^c$. Let $\Omega=\Omega_k,\varphi=h_n$ in Lemma \ref{lem2.2}, then we have
	\begin{align*}
		\sum_{x\in\Omega_k} &\mu(x) v(x)^{p-t+\epsilon} |\nabla v(x)|^q h_n(x)^s \nonumber\\
		&\lesssim \left( C_{p_1,t}\right) ^{\frac{p+q-t+\epsilon}{\epsilon}} (2s)^{\frac{mp^{\prime}+q+t(q-m)}{p^{\prime}+q-m+1}} t^{-\frac{p^{\prime}(m-1)+t(q-m+1)}{p^{\prime}+q-m+1}} \sum_{x,y\in \Omega_k} \mu_{xy} |\nabla_{xy} h_n|^{\frac{mp^{\prime}+q+t(q-m)}{p^{\prime}+q-m+1}} \nonumber\\
		&= \left( C_{p_1,t}\right) ^{\frac{p+q-t+\epsilon}{\epsilon}} (2s)^{\frac{mp+q+t(q-m)+m\epsilon}{\epsilon}} t^{-\frac{p(m-1)+t(q-m+1)+(m-1)\epsilon}{\epsilon}} \sum_{x,y\in \Omega_k} \mu_{xy} |\nabla_{xy} h_n|^{\frac{mp+q+t(q-m)+m\epsilon}{\epsilon}}  \\
		&\leqslant (2(C_{p_1,t}^{\prime}) (z+1))^z (\frac{1}{t})^z \sum_{x,y\in \Omega_k} \mu_{xy} |\nabla_{xy} h_n|^z,\nonumber
	\end{align*}
    where \begin{align*}
    	C_{p_1,t} &= \frac{(m-1)p+t(q-m+1)+(m-1)\epsilon}{p+q-t+\epsilon}\\
    	&\times \frac{\sqrt{2p_1}(1+p_1^t)(\sqrt{2p_1}(1+p_1^t)p_1^{t+1})^{\frac{(m-1)p+t(q-m+1)+(m-1)\epsilon}{p+q-t+\epsilon}}}{2^{1+\frac{(m-1)p+t(q-m+1)+(m-1)\epsilon}{p+q-t+\epsilon}}}
    \end{align*}    and $C_{p_1,t}^{\prime} = \left( C_{p_1,t}\right) ^{\frac{p+q-t+\epsilon}{mp+q+t(q-m)+m\epsilon}}.$

   Letting  $\epsilon\rightarrow0$, one gets
    \begin{equation*}
	\begin{aligned}
		C_{p_1,t}\rightarrow C^{\prime} 
		&= \frac{(m-1)p+t(q-m+1)}{p+q-t} \frac{\sqrt{2p_1}(1+p_1^t)(\sqrt{2p_1}(1+p_1^t)p_1^{t+1})^{\frac{(m-1)p+t(q-m+1)}{p+q-t}}}{2^{1+\frac{(m-1)p+t(q-m+1)}{p+q-t}}} \\
		&= \frac{p[\sqrt{2p_1}(1+p_1^t)]^{p+1} p_1^{(t+1)p}}{2^{1+p}}.
	\end{aligned}
    \end{equation*}
Since $h_n=1$ in $B(o,n)$, we have
\begin{equation}
	\begin{aligned}\label{eq2.53}
		\sum_{x\in\Omega_k\cap B(o,n)} \mu(x) v(x)^{p-t+\epsilon} |\nabla v(x)|^q &\lesssim \left( \frac{2C^{\prime}(z+1)}{t}\right) ^z \sum_{x,y\in \Omega_k} \mu_{xy} |\nabla_{xy} h_n|^z \\
		&\lesssim \left( \frac{2C^{\prime}(z+1)}{tn}\right) ^z W(B(o,n)) \\
		&\lesssim \left( \frac{2C^{\prime}(z+1)}{tn}\right) ^z e^{2\kappa n}. 
	\end{aligned}
\end{equation}

    Let
    $$
    z=\theta n,
    $$
    where the positive constant $\theta$ is to be determined later.

    Letting $n\rightarrow\infty$ in \eqref{eq2.53},  we obtain
   
	\begin{align}\label{eq2.54}
		\sum_{x\in\Omega_k\cap B(o,n)} \mu(x) v(x)^{p-t} |\nabla v(x)|^q
		&\lesssim \lim_{n\rightarrow\infty} \left( \frac{2C^{\prime}(\theta n+1)}{tn}\right) ^{\theta n} e^{2\kappa n} \nonumber\\
		&\lesssim \lim_{n\rightarrow\infty} e^{n(2\kappa+\theta \ln(\frac{2C^{\prime}\theta}{t}))}.
	\end{align}
   
    Similar to the first case, if $\kappa$ satisfies
    $$
    0<\kappa<\kappa_{0,t} = \frac{t}{4C^{\prime} e},
    $$
    there  exists $\theta>0$ such that $2\kappa+\theta \ln(\frac{2C\prime\theta}{t})<0$. Note that as long as $t\in (o,m-1)$ is chosen, $\kappa,\theta$ can be selected to satisfy the above conditions.

    From \eqref{eq2.54}, we arrive at
    $$
    \sum_{x\in\Omega_k} \mu(x) v(x)^{p-t} |\nabla v(x)|^q = 0
    $$
   From this and \eqref{eq2.51}, we get a contradiction. This completes the proof of Theorem $\ref{Thm1.1}$ (V).
\end{proof}

\section{Proof of Theorem \ref{Thm1.4}}\label{sec Thm2.4} \rm

We are ready to give  the proof of Theorem \ref{Thm1.4}. 
\begin{proof}
We complete the proof by using contradiction argument. Assume that $u$ is a nontrivial positive solution of $\eqref{eq1}$. Fix a vertex $x\in V$ and let us introduce two disjoint vertices sets $S_+(x)$
and $S_-(x)$ as follows
\begin{align*}
S_+(x) = \{y \sim x \mid u(y) < u(x)\},\quad
S_-(x) = \{y \sim x \mid u(y) > u(x)\}.
\end{align*}
Denote
$$\mu_+ = \sum_{y \in S_+} \mu_{xy},\quad \mu_- = \sum_{y \in S_-} \mu_{xy}$$
It follows that $\mu(x) \geq \mu_+ + \mu_-$.

Note
\begin{equation}\label{eq2}
	\begin{aligned}
		-\Delta_m u(x)
		&=\frac{1}{\mu (x)} \sum_{y\sim x}\mu_{xy} |u(y)-u(x)|^{m-2}[u(y)-u(x)]\\
		&=\frac{\sum_{y\in S_+} \mu_{xy}(u(x)-u(y))^{m-1}-\sum_{y\in S_-} \mu_{xy}(u(y)-u(x))^{m-1}}{\mu (x)} \\
		&\leq \frac{\mu_+ A_+ ^{m-1}- \mu_- A_- ^{m-1}}{\mu_+ + \mu_-},
	\end{aligned}
\end{equation}
where
\begin{align*}
	A_+ = [\frac{1}{\mu_+}\sum_{y\in S_+} \mu_{xy}(u(x)-u(y))^{m-1}]^{\frac{1}{m-1}}, A_- = [\frac{1}{\mu_-}\sum_{y\in S_-} \mu_{xy}(u(y)-u(x))^{m-1}]^{\frac{1}{m-1}}.
\end{align*}
%Since $m \ge 3$, the exponent $\frac{m-1}{2} \ge 1$. The function $f(t) = t^{(m-1)/2}$ is strictly convex for $t \ge 0$. 
By using Jensen's Inequality when $m\geq 3$, we have
\begin{equation*}
	\begin{aligned}
		|\nabla u(x)|^2
		&=\sum_{y\sim x}{\frac{\mu _{xy}}{2\mu (x)}(u(y)-u(x))^2} \\
		&= \frac{\sum_{y\in S_+} \mu_{xy}(u(x)-u(y))^2+\sum_{y\in S_-} \mu_{xy}(u(y)-u(x))^2}{2\mu (x)} \\
		&\leq \frac{\mu_+ A_+ ^2+ \mu_- A_- ^2}{2(\mu_+ + \mu_-)}.
	\end{aligned}
\end{equation*}

Since $q<0$, we have 
\begin{equation}\label{eq3}
	|\nabla u(x)|^q\geq \left(\frac{\mu_+ A_+ ^2+ \mu_- A_- ^2}{2(\mu_+ + \mu_-)}\right)^\frac{q}{2} .
\end{equation}

We next consider the case where $\mu_->0$. Substituting  $\eqref{eq2}$ and $\eqref{eq3}$ into $\eqref{eq1}$, it is necessary to obtain
\begin{align}\label{eq4}
	\underbrace{\frac{K-R^{m-1}}{1+K}}_{\text{Term 1}}\cdot \underbrace{\left(\frac{K+R^2}{2(1+K)}\right)^{-\frac{q}{2}}}_{\text{Term 2}} \cdot \underbrace{(\frac{A_+}{u})^p}_{\text{Term 3}} \geq 1,
\end{align}
where $K=\frac{\mu_+}{\mu_-}, R=\frac{A_-}{A_+}$.

Now we give an estimate the three terms on the left-hand side of $\eqref{eq4}$:
\begin{enumerate}
	\item For Term 1: To ensure $-\Delta_m u(x) > 0$, we must have $R^{m-1} < K$. Therefore, $\frac{K - R^{m-1}}{K+1} < \frac{K}{K+1} < 1$.
	\item For Term 2: Given $R^{m-1} < K$ and $m \ge 3$:
	\begin{itemize}
		\item If $R \ge 1$, then $R^2 \le R^{m-1} < K$, and hence $\frac{K+R^2}{2(K+1)} < \frac{2K}{2K+2} < 1$.
		\item If $R < 1$, then $R^2 < 1$, and hence $\frac{K+R^2}{2(K+1)} < \frac{K+1}{2K+2} = \frac{1}{2} < 1$.
	\end{itemize}
	\item For Term 3: For any $y \in S_+(x)$, since $u(y) > 0$, we have $ u(x) - u(y) < u(x)$. Thus, $A_+ = [\frac{1}{\mu_+}\sum_{y\in S_+} \mu_{xy}(u(x)-u(y))^{m-1}]^{\frac{1}{m-1}}<[\frac{1}{\mu_+}\sum_{y\in S_+} \mu_{xy}u(x)^{m-1}]^{\frac{1}{m-1}} =u(x)$, which implies $\frac{A_+}{u(x)} < 1$. 
\end{enumerate}
Hence, the left hand side of \eqref{eq4} is smaller than one and we obtain a contradiction.

We now consider $\mu_-=0$. Then, instead of \eqref{eq4}, we get 
$$\left(\frac{1}{2}\right)^{\frac{q}{2}}\leq \left(\frac{A_+}{u(x)}\right)^p.$$
As above, we also have $\frac{A_+}{u(x)} < 1$. In addition, since $q<0$, it follows that $\left(\frac{1}{2}\right)^{\frac{q}{2}}>1$. Thus we again obtain  a contradiction. The proof is completed. 
\end{proof}

\section{Proof of Theorem \ref{Thm1.2}}\label{sec-thm2}
 In this section, we give the proof of Theorem \ref{Thm1.2}. Recall that we only consider $q\not=0$ and $(p,q)\in G_6$.

	Assume that $u$ is a non-trivial positive solution of \eqref{eq1}.  Since $u$ is non-trivial, there is $x_0$ such that $|\nabla u(x_0)|>0.$ It follows that $-\Delta_{m}u(x_0)\geq u^p(x_0)|\nabla u(x_0)|^q>0$ or  $\Delta_{m}u(x_0)<0$. Let $x_1$ be such that $$u(x_1)=\min_{x\sim x_0}u(x).$$
	Then we also have 
	$$|\nabla u(x_1)|>0, \Delta_{m}u(x_1)<0 \mbox{ and }u(x_1)\leq u(x_0).$$
	By induction, there is a sequence $(x_n)$ such that  $u(x_{n+1})=\min_{x\sim x_n}u(x)$ and 
	$$|\nabla u(x_n)|>0, \Delta_{m}u(x_n)<0 \mbox{ and }u(x_{n+1})\leq u(x_n)\mbox{ for all }n\geq 0.$$
	This implies that the sequence $(u(x_n))$ is decreasing and positive. Hence 
	\begin{equation}\label{emqt81}
		\lim_{n\to+\infty}u(x_n)\geq 0
	\end{equation}
	and 
	\begin{align}\label{emqt82}
		\begin{split}
			0>\Delta_{m} u(x_n)&=\sum_{y\sim x_n}\frac{\mu_{x_ny}}{\mu(x_n)}|u(y)-u(x_n)|^{m-2}(u(y)-u(x_n))\\
			&\geq |u(x_{n+1})-u(x_n)|^{m-2}(u(x_{n+1})-u(x_n))\to 0\mbox{ as }n\to\infty.
		\end{split}
	\end{align}
	Consequently, $\Delta_{m} u(x_n)\to 0$ as $n\to\infty$. Since $u$ is a positive solution of \eqref{eq1}, we deduce that 
	\begin{equation}\label{emqt91}
		u^p(x_n)|\nabla u(x_n)|^q\to 0\mbox{ as }n\to\infty.
	\end{equation}
	We next consider two possible cases of $(p,q)$.
	
	 Case 1: $q<0, p<m-1-q$

    By the  condition \eqref{p_0} and \eqref{emqt82}, we obtain, for such vertex $y^*$ satisfying $y^*\sim x_n$ and $u(y^*)\geq u(x_n)$, that
	\begin{align*}%\label{emqt82}
		\begin{split}
			\frac{1}{p_0}(u(y^*)-u(x_n))^{m-1}	&\leq \sum_{\substack{y\sim x_n\\u(y)\geq u(x_n)}}\frac{\mu_{x_ny}}{\mu(x_n)}|u(y)-u(x_n)|^{m-2}(u(y)-u(x_n))\\
			&\leq \sum_{\substack{y\sim x_n\\ u(y)< u(x_n)}}\frac{\mu_{x_ny}}{\mu(x_n)}|u(y)-u(x_n)|^{m-2}(u(x_n)-u(y))\\
			&\leq (u(x_n)-u(x_{n+1}))^{m-1},
		\end{split}
	\end{align*}
	which means that
	\begin{equation}
		u(y^*)-u(x_n)\leq p_0^{\frac{1}{m-1}}(u(x_n)-u(x_{n+1})).
	\end{equation}
	Hence, 
	\begin{align*}
		2|\nabla u(x_n)|^2&=\sum_{\substack{y \sim x_n\\ u(y)\geq  u(x_n)}}\frac{\mu_{x_ny}}{\mu(x_n)}(u(y)-u(x_n))^2\\
		&+\sum_{\substack{y \sim x_n\\ u(y)< u(x_n)} }\frac{\mu_{x_ny}}{\mu(x_n)}(u(x_n)-u(y))^2\\
		&\leq (p_0^\frac{2}{m-1}+1)(u(x_n)-u(x_{n+1}))^2.
	\end{align*}
    Therefore, we obtain 
    \begin{equation}\label{eq4.2}
    	\lim_{n\rightarrow\infty} |\nabla 
    	u(x_n)| = 0.
    \end{equation}
    From \eqref{eq4.2} and \eqref{emqt91}, we deduce that $p>0$ and
    \begin{equation}\label{eq4.3}
    	\lim_{n\rightarrow \infty} u(x_n)=0.
    \end{equation}
	On the other hand, using \eqref{emqt82}, we get 
	\begin{equation}\label{emqt94}
		-\Delta_{m} u(x_n)\leq (u(x_n)-u(x_{n+1}))^{m-1}
	\end{equation}
	Substituting these estimates into \eqref{eq1} and noticing that $q<0$, we obtain 
	\begin{align*}
		u(x_n)^{p+q}&\leq u(x_n)^p(u(x_n)-u(x_{n+1}))^q\leq C u(x_n)^p|\nabla u(x_n)|^q\\
	&\leq -C\Delta_{m} u(x_n)\leq C(u(x_n)-u(x_{n+1}))^{m-1}\leq Cu(x_n)^{m-1}
	\end{align*}
	which gives 
	\begin{equation}\label{emqt93}
		u(x_n)^{p+q-m+1}\leq C,
	\end{equation}
	where $C$ depends only on $p_0$ and $q$. From the assumption $p+q-m+1<0$ and \eqref{eq4.3}, we deduce that $u(x_n)^{p+q-m+1}\to\infty$ as $n\to\infty$, which contradicts with \eqref{emqt93}.

	 Case 2: $0<q<m-1$ and $p<m-1-q$

	By the condition \eqref{p_0}, we have 
	$$2|\nabla u(x_n)|^2\geq \frac{1}{p_0}(u(x_n)-u(x_{n+1}))^2.$$
	In  this case, we still have  \eqref{emqt94} which comes from \eqref{emqt82}. Combining these estimates and \eqref{eq1}, we arrive at
	\begin{align*}
		u(x_n)^p(u(x_n)-u(x_{n+1}))^q&\leq Cu(x_n)^p|\nabla u(x_n)|^q\\
		&\leq -C\Delta_{m}u(x_n)\\
		&\leq C(u(x_n)-u(x_{n+1}))^{m-1}.
	\end{align*}
	Consequently,
	\begin{equation}\label{emqt95}
		1\leq Cu(x_n)^{-p}(u(x_n)-u(x_{n+1}))^{m-1-q},
	\end{equation}
	where $C$ only depends on $p_0$ and $q$. From \eqref{emqt95} and the fact that $u(x_n)-u(x_{n+1})\to 0$ as $n\to\infty$, we deduce that $u(x_n)\to 0$ as $n\to\infty$ and $p>0$. It implies that $u(x_n)^{-p}\leq (u(x_n)-u(x_{n+1}))^{-p}$. Combining this and \eqref{emqt95}, we obtain 
	$$1\le C(u(x_n)-u(x_{n+1}))^{m-1-q-p}.$$
	Recall that $m-1-q-p>0$. Letting $n\to\infty$ in this inequality, we get  a contradiction. The proof is finished.\qed

\section{Proof of Theorem \ref{Thm1.3}}\label{sec-thm3}

Before proving Theorem \ref{Thm1.3}, we first introduce some notations. For any fixed $n\geqslant0$, let $D_n=\{ x\in T_N : d(o,x)=n\}$, and let $E_n$ denote the set of all the edges from vertices in $D_n$ to vertices in $D_{n+1}$.

\begin{proof}[Proof of Theorem \ref{Thm1.3} (I)]
    When $(p,q)\in G_1$, define 
    \begin{equation}
    \begin{aligned}
        &\mu_{xy}=\mu_n=\frac{(n+n_0)^{\frac{(m-1)(p+1)}{p+q-m+1}}(\ln(n+n_0))^{\frac{m-1}{p+q-m+1}+\epsilon}}{(N-1)^n},\mbox{ for }(x,y)\in E_n, n\geq 0, 
    \end{aligned}\label{eq3.1}
    \end{equation}
and 
    \begin{equation}\label{eq3.2}
        u(x)=u_n=\frac{\delta}{(n+n_0)^\frac{m-q}{p+q-m+1}(\ln(n+n_0))^\frac{1}{p+q-m+1}},\mbox{ for } x\in D_n, n\geq 0,
    \end{equation}
    where $n_0\geqslant2$ and $\delta>0$ are to be determined later.

    For convenience,  let us denote
    \[
    \lambda=\frac{(m-1)(p+1)}{p+q-m+1},\text{ }\beta=\frac{1}{p+q-m+1},\text{ } \sigma=\frac{m-q}{p+q-m+1},\text{ } \beta^{\prime}=(m-1)\beta.
    \]

    From \eqref{eq3.1}, we have
    \[
    W_o(n)=\sum_{k=0}^{n} \mu (D_k)\asymp \sum_{k=0}^{n}(N-1)^k \mu_k \asymp n^{\frac{mp+q}{p+q-m+1}}(\ln n)^{\frac{m-1}{p+q-m+1}+\epsilon},
    \]
     namely \eqref{eq1.9} is satisfied.

    We next prove that  \eqref{eq1} holds, which means that the following two inequalities hold:
    \begin{equation}\label{eq3.3}
        -(u_0-u_1)^{m-1}+u_0^p\left[\frac{(u_0-u_1)^2}{2}\right]^{\frac{q}{2}}\leqslant0,
    \end{equation}
    and for all $n\geq 1$
    \begin{equation}
    	\begin{aligned}\label{eq3.4}
    		&\frac{-(N-1)\mu_n(u_n-u_{n+1})^{m-1}+\mu_{n-1}(u_{n-1}-u_n)^{m-1}}{(N-1)\mu_n+\mu_{n-1}} \\
    		& \quad + u_n^p\Big[\frac{(N-1)\mu_n(u_n-u_{n+1})^2+\mu_{n-1}(u_{n-1}-u_n)^2}{2(N-1)\mu_n+2\mu_{n-1}}\Big]^\frac{q}{2}\leqslant0.
    	\end{aligned} 
    \end{equation} 
 We first prove \eqref{eq3.3}. Plugging \eqref{eq3.1} and \eqref{eq3.2} into \eqref{eq3.3}, we have
    \begin{align*}
          -\Big(\frac{\delta}{n_0^\sigma(\ln n_0)^\beta}-\frac{\delta}{(n_0+1)^\sigma(\ln (n_0+1))^\beta}\Big)^{m-1} + \Big(\frac{\delta}{n_0^\sigma(\ln n_0)^\beta}\Big)^p \\
          \times \Big[\frac{1}{2}(\frac{\delta}{n_0^\sigma(\ln n_0)^\beta}-\frac{\delta}{(n_0+1)^\sigma(\ln (n_0+1))^\beta})\Big]^\frac{q}{2} \leqslant 0.
    \end{align*}
    The above inequality holds if we take $\delta\leqslant\delta_0$, where 
    \begin{align*}
         \delta_0=&2^\frac{q}{2(p+q-m+1)}
         (n_0^\sigma(\ln n_0)^\beta)^\frac{p}{p+q-m+1}  \left(\frac{1}{n_0^\sigma(\ln n_0)^\beta}-\frac{1}{(n_0+1)^\sigma(\ln (n_0+1))^\beta}\right)^\frac{m-1-q}{p+q-m+1} .
    \end{align*}

 We now prove \eqref{eq3.4}. Substituting \eqref{eq3.1} and \eqref{eq3.2} into \eqref{eq3.4}, we get
    \begin{align*}
        &\frac{-(n+n_0)^\lambda(\ln (n+n_0))^{\beta^{\prime}+\epsilon}  }{(n+n_0)^\lambda(\ln (n+n_0))^{\beta^{\prime}+\epsilon}+(n+n_0-1)^\lambda(\ln (n+n_0-1))^{\beta^{\prime}+\epsilon}} \\
        &\times \Big(\frac{\delta}{(n+n_0)^\sigma(\ln (n+n_0))^\beta}-\frac{\delta}{(n+n_0+1)^\sigma(\ln (n+n_0+1))^\beta}\Big)^{m-1}\\
        &+ \frac{(n+n_0-1)^\lambda(\ln (n+n_0-1))^{\beta^{\prime}+\epsilon}}{(n+n_0)^\lambda(\ln (n+n_0))^{\beta^{\prime}+\epsilon}+(n+n_0-1)^\lambda(\ln (n+n_0-1))^{\beta^{\prime}+\epsilon}}\\
        &\times \Big(\frac{\delta}{(n+n_0-1)^\sigma(\ln (n+n_0-1))^\beta}-\frac{\delta}{(n+n_0)^\sigma(\ln (n+n_0))^\beta}\Big)^{m-1}\\
       &+ \Big(\frac{\delta}{(n+n_0)^\sigma(\ln (n+n_0))^\beta}\Big)^p \\
       &\times \Bigg[ \frac{-(n+n_0)^\lambda(\ln (n+n_0))^{\beta^{\prime}+\epsilon}\Big(\frac{\delta}{(n+n_0)^\sigma(\ln (n+n_0))^\beta}-\frac{\delta}{(n+n_0+1)^\sigma(\ln (n+n_0+1))^\beta}\Big)^{2} }{2(n+n_0)^\lambda(\ln (n+n_0))^{\beta^{\prime}+\epsilon}+2(n+n_0-1)^\lambda(\ln (n+n_0-1))^{\beta^{\prime}+\epsilon}} \\
        &+ \frac{(n+n_0-1)^\lambda(\ln (n+n_0-1))^{\beta^{\prime}+\epsilon}}{2(n+n_0)^\lambda(\ln (n+n_0))^{\beta^{\prime}+\epsilon}+2(n+n_0-1)^\lambda(\ln (n+n_0-1))^{\beta^{\prime}+\epsilon}}\\
        &\times \Big(\frac{\delta}{(n+n_0-1)^\sigma(\ln (n+n_0-1))^\beta}-\frac{\delta}{(n+n_0)^\sigma(\ln (n+n_0))^\beta}\Big)^{2}\Bigg]^\frac{q}{2}  \leqslant0.
    \end{align*}
    The above inequality is equivalent to    \begin{equation}\label{eq3.5}
        \delta^{p+q-m+1}\leqslant\Lambda_1(n+n_0),\;n\geq 1
    \end{equation}
    where
    \begin{equation}
    	\begin{aligned}\label{eq3.6}
    		\Lambda_1(n)&\coloneqq  2^{\frac{q}{2}} (n^\sigma(\ln n)^\beta)^{p-m+1} \Bigg(\Big[1-(1-\frac{1}{n+1})^\sigma(\frac{\ln n}{\ln (n+1)})^\beta \Big]^{m-1}\\
    		&-\Big[(1-\frac{1}{n})(\frac{\ln (n-1)}{\ln n})^{ \frac{\epsilon}{m-1}}-(1-\frac{1}{n})^{\frac{\lambda}{m-1}}(\frac{\ln (n-1)}{\ln n})^{\beta+{ \frac{\epsilon}{m-1}}}\Big]^{m-1}\Bigg)\\
    		&\times\Big[1+(1-\frac{1}{n})^\lambda(\frac{\ln (n-1)}{\ln n})^{\beta^\prime+\epsilon}\Big]^{-1}\Bigg[\frac{n^\lambda(\ln n)^{\beta^\prime+\epsilon}(\frac{1}{n^\sigma(\ln n)^\beta} - \frac{1}{(n+1)^\sigma(\ln (n+1)^\beta})^2 }{n^\lambda(\ln n)^{\beta^\prime+\epsilon}+(n-1)^\lambda(\ln (n-1))^{\beta^\prime+\epsilon}}\\
    		&+ \frac{ (n-1)^\lambda(\ln (n-1))^{\beta^\prime+\epsilon}(\frac{1}{(n-1)^\sigma(\ln (n-1))^\beta}-\frac{1}{n^\sigma(\ln n)^\beta})^2}{n^\lambda(\ln n)^{\beta^\prime+\epsilon}+(n-1)^\lambda(\ln (n-1))^{\beta^\prime+\epsilon}}\Bigg]^{-\frac{q}{2}}.
    	\end{aligned}
    \end{equation}
    For convenience, set
        \begin{align}\label{eq3.7}
           A_1(n)=&\Bigg[\frac{n^\lambda(\ln n)^{\beta^\prime+\epsilon}(\frac{1}{n^\sigma(\ln n)^\beta} - \frac{1}{(n+1)^\sigma(\ln (n+1)^\beta})^2 }{n^\lambda(\ln n)^{\beta^\prime+\epsilon}+(n-1)^\lambda(\ln (n-1))^{\beta^\prime+\epsilon}}\nonumber\\
            &+ \frac{ (n-1)^\lambda(\ln (n-1))^{\beta^\prime+\epsilon}(\frac{1}{(n-1)^\sigma(\ln (n-1))^\beta}-\frac{1}{n^\sigma(\ln n)^\beta})^2}{n^\lambda(\ln n)^{\beta^\prime+\epsilon}+(n-1)^\lambda(\ln (n-1))^{\beta^\prime+\epsilon}}\Bigg]^{-\frac{q}{2}}.
        \end{align}
 We  claim that
    \begin{equation}\label{eq3.8}
        \lim_{n\rightarrow\infty} \Lambda_1(n) =  2^{\frac{q}{2}-1}\sigma^{m-1-q}\epsilon.
    \end{equation}
   Indeed, by using  Taylor's expansion:
    \[
    \frac{\ln (n-1)}{\ln n} = 1-\frac{1}{n\ln n}-\frac{1}{2n^2\ln n}+o(\frac{1}{2n^2\ln n})
    \]
    and
    \[
    (1-\frac{1}{n})^\alpha=1-\frac{\alpha}{n}+\frac{\alpha(\alpha-1)}{2n^2}+o(\frac{1}{n^2}),
    \]
    we obtain
    \begin{equation}
    	\begin{aligned}\label{eq3.9}
    		A_1(n)= &\left(1+(1-\frac{1}{n})^\lambda(\frac{\ln (n-1)}{\ln n})^{\beta+\epsilon}\right)^{\frac{q}{2}}\Bigg[\left(\frac{1}{n^\sigma(\ln n)^\beta} - \frac{1}{(n+1)^\sigma\ln (n+1)^\beta}\right)^2\\
    		&+(1-\frac{1}{n})^\lambda(\frac{\ln (n-1)}{\ln n})^{\beta+\epsilon}\left(\frac{1}{(n-1)^\sigma(\ln (n-1))^\beta} - \frac{1}{n^\sigma(\ln n)^\beta}\right)^2\Bigg]^{-\frac{q}{2}}.
    	\end{aligned}
    \end{equation}
    Note that
    \begin{equation}
    	\begin{aligned}\label{eq3.10}
    		\frac{1}{n^\sigma(\ln n)^\beta} -& \frac{1}{(n+1)^\sigma\ln (n+1)^\beta}         
    		= \frac{1}{n^\sigma(\ln n)^\beta} \left( 1-(1-\frac{1}{n+1})^\sigma(\frac{\ln n}{\ln (n+1)})^\beta \right) \\
    		&=\frac{1}{n^\sigma(\ln n)^\beta} \left( \frac{\sigma}{n+1} + \frac{\beta}{(n+1)\ln (n+1)} + o(\frac{1}{n \ln n}) \right) \\
    		& =\frac{1}{n^\sigma(\ln n)^\beta} \left( \frac{\sigma}{n}+o(\frac{1}{n}) \right)
    	\end{aligned}
    \end{equation}
and 
    \begin{equation}\label{eq3.11}
        \frac{1}{(n-1)^\sigma\ln (n-1)^\beta} - \frac{1}{n^\sigma(\ln n)^\beta} = \frac{1}{n^\sigma(\ln n)^\beta} \left( \frac{\sigma}{n}+o(\frac{1}{n}) \right).
    \end{equation}
    Substituting \eqref{eq3.11} and \eqref{eq3.10} into \eqref{eq3.9}, we obtain
    \begin{equation}\label{eq3.12}
        A_1(n) =\left( \frac{n^{\sigma+1}(\ln n)^\beta}{\sigma} \right)^q(1+o(1)).
    \end{equation}
   It follows from  \eqref{eq3.12} and  \eqref{eq3.6} that
    \begin{equation*}
        \begin{aligned}
            \Lambda_1(n)
            & =(\frac{2^\frac{1}{2}}{\sigma})^q n^q (n^\sigma(\ln n)^\beta)^{p+q-m+1}\left(1+(1-\frac{1}{n})^\lambda(\frac{\ln (n-1)}{\ln n})^{\beta^\prime+\epsilon}\right)^{-1}\\
            &\times\Bigg( [1-(1-\frac{1}{n+1})^\sigma(\frac{\ln n}{\ln (n+1)})^\beta]^{m-1} \\
            &- [(1-\frac{1}{n})(\frac{\ln (n-1)}{\ln n})^{ \frac{\epsilon}{m-1}}
             - (1-\frac{1}{n})^\frac{\lambda}{m-1}(\frac{\ln (n-1)}{\ln n})^{\beta+{ \frac{\epsilon}{m-1}}}]^{m-1} \Bigg) \\
            &=(\frac{2^\frac{1}{2}}{\sigma})^q n^m \ln n \left(1+(1-\frac{1}{n})^\lambda(\frac{\ln (n-1)}{\ln n})^{\beta^\prime+\epsilon}\right)^{-1}\\
            &\times\Bigg([1-(1-\frac{1}{n+1})^\sigma(\frac{\ln n}{\ln (n+1)})^\beta]^{m-1}\\
            & - [(1-\frac{1}{n})(\frac{\ln (n-1)}{\ln n})^{ \frac{\epsilon}{m-1}} - (1-\frac{1}{n})^\frac{\lambda}{m-1}(\frac{\ln (n-1)}{\ln n})^{\beta+{ \frac{\epsilon}{m-1}}}]^{m-1} \Bigg),
        \end{aligned}
    \end{equation*}
   Here, recall that
    $$\lambda=\frac{(m-1)(p+1)}{p+q-m+1},\beta=\frac{1}{p+q-m+1},$$
    and $$ \sigma=\frac{m-q}{p+q-m+1}=\frac{\lambda}{m-1}-1=\lambda'-1.$$

    By applying the Taylor expansion, we have
        \begin{align}\label{eq3.13}
              \Lambda_1(n)&=\frac{2^{\frac{q}{2}-1}}{\sigma^q} n^m \ln n \nonumber\\
            & \times \bigg\{\Big[\frac{\sigma}{n+1}+\frac{\beta}{(n+1)\ln (n+1)}-\frac{\sigma(\sigma-1)}{2(n+1)^2}+\frac{\frac{\beta}{2}-\beta\sigma}{(n+1)^2\ln(n+1)}+o(\frac{1}{n^2 \ln n})\Big]^{m-1} \nonumber\\
            &- \Big[\frac{\sigma}{n}+\frac{\beta}{n\ln n}-\frac{\sigma(\sigma+1)}{2n^2}+\frac{(\frac{1}{2}-\lambda')\beta-\sigma{ \frac{\epsilon}{m-1}}}{n^2 \ln n} +o(\frac{1}{n^2 \ln n})\Big]^{m-1}   \bigg\}.  
        \end{align}
    Using Taylor's expansion again, we obtain
    \begin{equation}
    	 \begin{aligned}\label{eq3.14}
    		\bigg[&\frac{\sigma}{n+1}+\frac{\beta}{(n+1)\ln (n+1)}-\frac{\sigma(\sigma-1)}{2(n+1)^2}+\frac{\frac{\beta}{2}-\beta\sigma}{(n+1)^2\ln(n+1)}+o(\frac{1}{n^2 \ln n})\bigg]^{m-1} \\
    		&= \frac{\sigma^{m-1}}{(n+1)^{m-1}} \bigg[1+\frac{\beta}{\sigma\ln (n+1)}-\frac{\sigma-1}{2(n+1)}-\frac{\frac{\beta}{2\sigma}-\beta}{(n+1)\ln(n+1)}+o(\frac{1}{n\ln n})\Big]^{m-1}\\
    		&= \frac{\sigma^{m-1}}{(n+1)^{m-1}} \bigg[ 1+ \sum_{k=1}^{\infty} \frac{(m-1)(m-2)\cdots(m-k)}{k!} \\
    		&\times \Big(\frac{\beta}{\sigma\ln (n+1)}-\frac{\sigma-1}{2(n+1)}
    		-\frac{\frac{\beta}{2\sigma}-\beta}{(n+1)\ln(n+1)}+o(\frac{1}{n\ln n})\Big)^k \bigg] \\
    		&= \frac{\sigma^{m-1}}{(n+1)^{m-1}} \Big[ 1+(m-1)\Big(\frac{\beta}{\sigma\ln (n+1)}-\frac{\sigma-1}{2(n+1)}-\frac{\frac{\beta}{2\sigma}-\beta}{(n+1)\ln(n+1)}\Big)  \\
    		&\qquad -\frac{(m-1)(m-2)}{2} \frac{\beta(\sigma-1)}{\sigma(n+1)\ln (n+1)} \\
    		&\qquad + \sum_{k=2}^{\infty} \frac{(m-1)(m-2)\cdots(m-k)}{k!} \Big(\frac{\beta}{\sigma\ln (n+1)}\Big)^k \Big] + o(\frac{1}{n^m\ln n}).
    	\end{aligned}
    \end{equation}
    Similarly
    \begin{equation}\label{eq3.15}
        \begin{aligned}
            \Big[\frac{\sigma}{n}+&\frac{\beta}{n\ln n}-\frac{\sigma(\sigma+1)}{2n^2}+\frac{(\frac{1}{2}-\lambda')\beta-\sigma{ \frac{\epsilon}{m-1}}}{n^2 \ln n} +o(\frac{1}{n^2 \ln n})\Big]^{m-1} \\
            &=\frac{\sigma^{m-1}}{n^{m-1}} \Big[1+\frac{\beta}{\sigma \ln n}-\frac{\sigma+1}{2n}+\frac{(\frac{1}{2}-\lambda')\frac{\beta}{\sigma}-{ \frac{\epsilon}{m-1}}}{n\ln n} \Big]^{m-1}\\
            &=\frac{\sigma^{m-1}}{n^{m-1}} \Big[ 1+(m-1)(\frac{\beta}{\sigma \ln n}-\frac{\sigma+1}{2n}+\frac{(\frac{1}{2}-\lambda')\frac{\beta}{\sigma}-{ \frac{\epsilon}{m-1}}}{n\ln n}) \\
            & \qquad- \frac{(m-1)(m-2)}{2}\frac{\beta(\sigma+1)}{\sigma n \ln n} \\
            &\qquad+ \sum_{k=2}^{\infty} \frac{(m-1)(m-2)\cdots(m-k)}{k!} \Big(\frac{\beta}{\sigma\ln n}\Big)^k   \Big]+ o(\frac{1}{n^m\ln n}).
        \end{aligned}
    \end{equation}
    Substituting \eqref{eq3.14} and \eqref{eq3.15} into \eqref{eq3.13}, we have
    \[
    \lim_{n\rightarrow\infty} \Lambda_1(n) =  2^{\frac{q}{2}-1}\sigma^{m-1-q}\epsilon.
    \]
    Thus, we obtain \eqref{eq3.8}. This means that there exists a sufficiently large $n_0$ such that for all $n\geqslant n_0$, the RHS of \eqref{eq3.5} is bounded from  below by $\frac{2^{\frac{q}{2}-1}\sigma^{m-1-q}\epsilon}{2}$. 
    
Hence, taking  $\delta=\min\{\delta_0,\delta_1\}$, where $ \delta_1=\Big(\frac{2^{\frac{q}{2}-1}\sigma^{m-1-q}\epsilon}{2}\Big)^\frac{1}{p+q-m
    +1}$, we deduce that \eqref{eq3.5} is true.               
\end{proof}

\begin{proof}[Proof of Theorem \ref{Thm1.3} (II)]
    In the case where $(p, q) \in G_2$, we construct 
   $$ \mu_{x y}= \mu_n=\frac{\left(n+n_0\right)^{m-1}\left(\ln \left(n+n_0\right)\right)^{m-1+\epsilon}}{(N-1)^n},\mbox{ for any }(x, y) \in E_n, n \geq 0,$$
   and 
   $$
   u(x)=u_n=\frac{1}{\left(\ln \left(n+n_0\right)\right)^{\frac{\epsilon}{2}}}+1,\mbox{ for any }x \in D_n, n \geq 0 .
   $$
   Then for $n \geq 2$, it is not difficult to verify that
   $$
   W_o(n) \asymp n^m(\ln n)^{m-1+\epsilon},
   $$
   which satisfies \eqref{eq1.10}.
   
   We next show that $u$ is a positive solution of \eqref{eq1}. 
   
   For $n=0$, we need to prove that 
   $$\begin{aligned}  -\left(\frac{1}{\left(\ln n_0\right)^{\frac{\epsilon}{2}}}-\frac{1}{\left(\ln \left(n_0+1\right)\right)^{\frac{\epsilon}{2}}}\right)^{m-1}&+\left(\frac{1}{\left(\ln n_0\right)^{\frac{\epsilon}{2}}}+1\right)^p \\ & \times\left[\frac{1}{2}\left(\frac{1}{\left(\ln n_0\right)^{\frac{\epsilon}{2}}}-\frac{1}{\left(\ln \left(n_0+1\right)\right)^{\frac{\epsilon}{2}}}\right)^2\right]^{\frac{q}{2}} \leq 0\end{aligned}$$,
   equivalently, 
   $$\begin{aligned} & \left(\frac{1}{\left(\ln n_0\right)^{\frac{\epsilon}{2}}}-\frac{1}{\left(\ln \left(n_0+1\right)\right)^{\frac{\epsilon}{2}}}\right)^{m-1}\\
   	&\times\Bigg[-1+\left(\frac{1}{\left(\ln n_0\right)^{\frac{\epsilon}{2}}}+1\right)^p 
   	\frac{1}{2^\frac{q}{2}}\left(\frac{1}{\left(\ln n_0\right)^{\frac{\epsilon}{2}}}-\frac{1}{\left(\ln \left(n_0+1\right)\right)^{\frac{\epsilon}{2}}}\right)^{q-m}\Bigg] \leq 0\end{aligned}$$
   With $n_0$ large enough and $q>m$, we see that 
   $$-1+\left(\frac{1}{\left(\ln n_0\right)^{\frac{\epsilon}{2}}}+1\right)^p 
   \frac{1}{2^\frac{q}{2}}\left(\frac{1}{\left(\ln n_0\right)^{\frac{\epsilon}{2}}}-\frac{1}{\left(\ln \left(n_0+1\right)\right)^{\frac{\epsilon}{2}}}\right)^{q-m}\leq 0.$$
   Thus when $n=0$, \eqref{eq1} is true.
   
   For $n\geq 1$, \eqref{eq3.4} becomes
   \begin{align} \label{emqt14}
   	\begin{split}
   		& \left((n+n_0)^{m-1}(\ln(n+n_0))^{m-1+\varepsilon}
   		+(n+n_0-1)^{m-1}(\ln(n+n_0-1))^{m-1+\varepsilon}\right)\\
   		&\times\left(\frac{1}{\left(\ln \left(n+n_0\right)\right)^{\frac{\epsilon}{2}}}+1\right)^p \\ & \times\left(\frac{\left(n+n_0\right)^{m-1}\left(\ln \left(n+n_0\right)\right)^{m-1+\epsilon}\left(\frac{1}{\left(\ln \left(n+n_0\right)\right)^{\frac{\epsilon}{2}}}-\frac{1}{\left(\ln \left(n+n_0+1\right)\right)^{\frac{\epsilon}{2}}}\right)^2}{2\left(n+n_0\right)^{m-1}\left(\ln \left(n+n_0\right)\right)^{m-1+\epsilon}+2\left(n+n_0-1\right)^{m-1}\left(\ln \left(n+n_0-1\right)\right)^{m-1+\epsilon}}\right. \\ & \left.+\frac{\left(n+n_0-1\right)^{m-1}\left(\ln \left(n+n_0-1\right)\right)^{m-1+\epsilon}\left(\frac{1}{\left(\ln \left(n+n_0-1\right)\right)^{\frac{\epsilon}{2}}}-\frac{1}{\left(\ln \left(n+n_0\right)\right)^{\frac{\epsilon}{2}}}\right)^2}{2\left(n+n_0\right)^{m-1}\left(\ln \left(n+n_0\right)\right)^{m-1+\epsilon}+2\left(n+n_0-1\right)^{m-1}\left(\ln \left(n+n_0-1\right)\right)^{m-1+\epsilon}}\right)^{\frac{q}{2}}\\
   		&\leq (n+n_0)^{m-1}(\ln(n+n_0))^{m-1+\epsilon}
   		\left(
   		\frac{1}{(\ln(n+n_0))^{\frac{\epsilon}{2}}}
   		-
   		\frac{1}{(\ln(n+n_0+1))^{\frac{\epsilon}{2}}}
   		\right)^{m-1}
   		\\
   		&
   		-(n+n_0-1)^{m-1}(\ln(n+n_0-1))^{m-1+\epsilon}
   		\left(
   		\frac{1}{(\ln(n+n_0-1))^{\frac{\epsilon}{2}}}
   		-
   		\frac{1}{(\ln(n+n_0))^{\frac{\epsilon}{2}}}
   		\right)^{m-1}.		
   	\end{split}
   \end{align}
   By using the expansions as in the first case, we obtain that \eqref{emqt14} is equivalent to 
   \begin{equation}\label{emqt21}
   	A_2(n+n_0)\leq B_2(n+n_0),
   \end{equation}
   or 
   \begin{equation}
   	1\leq \Lambda_2(n+n_0):=A_2^{-1}(n+n_0)B_2(n+n_0),
   \end{equation}
   where for $n$ large
   $$A_2(n)=\frac{\epsilon^q}{8^\frac{q}{2}}n^{m-1-q}(\ln n)^{m-1-q-\frac{\epsilon q}{2}+\epsilon}(2+o(1))$$
   and 
   $$B_2(n)=\frac{\epsilon^{m-1}(m-1)}{2^{m-1}}n^{-1}(\ln n)^{\epsilon-\frac{(m-1)\epsilon}{2}}(1+o(1)).$$
   Note that for $n$ large, we have
   \begin{align*}
   	\Lambda_2(n)=\frac{\epsilon^{m-1}(m-1)}{2^{m-1}}\frac{8^\frac{q}{2}}{\epsilon^q}n^{q-m}(\ln n)^{q-(m-1)+\frac{\epsilon}{2}(q-(m-1))}\frac{1+o(1)}{2+o(1)}>1,
   \end{align*}
   where we have used $q\geq m$ on $G_2$. Hence, \eqref{emqt14} is true for $n_0$ large and $n\geq 1$. We finish the proof.
\end{proof}

\begin{proof}[Proof of Theorem \ref{Thm1.3} (III)]
     When $(p,q)\in G_3$, we construct adge weight and positive solution of the form
     \begin{equation}\label{eq3.20}
         \mu_{xy}=\mu_n=\frac{(n+n_0)^\frac{m-1}{q-m+1}(\ln (n+n_0))^{\frac{m-1}{q-m+1}+\epsilon}}{(N-1)^n}, \text{ for any }(x,y)\in E_n,\text{ } n\geqslant0,
     \end{equation}

     \begin{equation}\label{eq3.21}
         u(x)=u_n=\frac{\delta}{(n+n_0)^\frac{m-q}{q-m+1}(\ln (n+n_0))^\frac{1}{q-m+1}}+1,\text{ for any }x\in D_n,\text{ } n\geqslant0,
     \end{equation}
     where $n_0\geqslant2 $ and $0<\delta<1$ are to be determined.  Under the above choice of $\mu$ that \eqref{eq1.11} holds. For brevity, we denote $\sigma=\frac{m-q}{q-m+1}$, $\beta=\frac{1}{q-m+1}$, $\beta'=(m-1)\beta$.

     Now let us consider the cases of $n=0$ and $n\geqslant1$.

   For $n=0$, substituting \eqref{eq3.20} and \eqref{eq3.21} into \eqref{eq3.3}, we have
     \begin{equation*}
         \begin{aligned}
             -\Big(\frac{\delta}{n_0^\sigma(\ln n_0)^\beta}-\frac{\delta}{(n_0+1)^\sigma(\ln (n_0+1))^\beta} \Big)^{m-1} + \Big( \frac{\delta}{n_0^\sigma(\ln n_0)^\beta} +1 \Big)^p  \\
             \times \Big[ \frac{1}{2}  \Big(\frac{\delta}{n_0^\sigma(\ln n_0)^\beta}-\frac{\delta}{(n_0+1)^\sigma(\ln (n_0+1))^\beta} \Big)^2\Big] \leqslant0.
         \end{aligned}
     \end{equation*}
Since $p<0$ and $q>m-1$, such $\delta$ can be obtained by choosing $\delta\leqslant\delta_0$, where \[\delta_0=2^\frac{q}{2(q-m+1)} \Big(\frac{1}{n_0^\sigma(\ln n_0)^\beta}+1\Big)^\frac{-p}{q-m+1} \Big(\frac{1}{n_0^\sigma(\ln n_0)^\beta} - \frac{1}{(n_0+1)^\sigma(\ln (n_0+1))^\beta} \Big)^{-1} .\]

For $n\geq 1$, plugging \eqref{eq3.20} and \eqref{eq3.21} into \eqref{eq3.4}, then \eqref{eq3.4} is equivalent to
     \begin{equation}\label{eq3.22}
         \delta^{q-m+1} \leqslant \Lambda_3(n+n_0) ,
     \end{equation}
     where
     \begin{equation*}
         \begin{aligned}
             &\Lambda_3(n) \coloneqq \Big(\frac{1}{n_0^\sigma(\ln n_0)^\beta}+1\Big)^{-p} 2^\frac{q}{2} \times\\
             & \Big[ \frac{\big( \frac{1}{n^\sigma(\ln n)^\beta}-\frac{1}{(n+1)^\sigma(\ln (n+1))^\beta}\big)^{m-1}-  (1-\frac{1}{n})^{\beta'} (\frac{\ln (n-1)}{\ln n})^{\beta'+\epsilon} \big( \frac{1}{(n-1)^\sigma(\ln (n-1))^\beta}-\frac{1}{n^\sigma(\ln n)^\beta}\big)^{m-1}}{1+ (1-\frac{1}{n})^{\beta'} (\frac{\ln (n-1)}{\ln n})^{\beta'+\epsilon}} \Big]  \\
             &\times\bigg[\frac{\big( \frac{1}{n^\sigma(\ln n)^\beta}-\frac{1}{(n+1)^\sigma(\ln (n+1))^\beta}\big)^2-(1-\frac{1}{n})^{\beta'} (\frac{\ln (n-1)}{\ln n})^{\beta'+\epsilon} \big( \frac{1}{(n-1)^\sigma(\ln (n-1))^\beta}-\frac{1}{n^\sigma(\ln n)^\beta}\big)^2}{1+ (1-\frac{1}{n})^{\beta'} (\frac{\ln (n-1)}{\ln n})^{\beta'+\epsilon}} \bigg]^{-\frac{q}{2}}
         \end{aligned}
     \end{equation*}

     Applying the same computation as used in Theorem \ref{Thm1.3} (I), we obtain
     \[ \lim_{n\rightarrow\infty} \Lambda_3(n) =   2^\frac{q-2}{2} \sigma^{m-1-q} \epsilon.\]
     This means that there exists some large $n_0$ such that for all $n\geqslant n_0$, the RHS of \eqref{eq3.22} is bounded  from below by $\frac{2^\frac{q-2}{2} \sigma^{m-1-q} \epsilon}{2}$. Consequently, for $n_0$ is sufficiently large and $\delta=\min\{\delta_0,\delta_1\}$, where $ \delta_1=\Big(\frac{ 2^\frac{q-2}{2} \sigma^{m-1-q} \epsilon}{2}\Big)^\frac{1}{q-m
    +1}$, \eqref{eq3.22}  is true.    The proof is completed.
\end{proof}

\begin{proof}[Proof of Theorem \ref{Thm1.3} (IV)]
    When $(p,q)\in G_4$, given $\lambda>0$, we construct
    \begin{equation}\label{eq3.23}
        \mu_{xy}=\mu_n=\frac{\lambda e^{\lambda(n+n_0)}}{(N-1)^n},\text{ for any }(x,y)\in E_n,\text{ } n\geqslant0,
    \end{equation}

    \begin{equation}\label{eq3.24}
        u(x)=u_n=\frac{\delta}{n+n_0} + \delta,\text{ for any }x\in D_n,\text{ } n\geqslant0,
    \end{equation}
    where $n_0\geqslant2$ and $\delta>0$ are to be determined later. It is  not hard to verify that \eqref{eq1.12} holds.

  For $n=0$, we need to prove that 
    \begin{equation*}
        -\big(\frac{\delta}{n_0}-\frac{\delta}{n_0+1} \big)^{m-1} + \big(\frac{\delta}{n_0}+\delta \big)^p \Big(\frac{1}{2} \big(\frac{\delta}{n_0}-\frac{\delta}{n_0+1} \big)^2 \Big)^\frac{m-1}{2} \leqslant0.
    \end{equation*}
    Since $p<0$, we can choose $\delta$ large enough to ensure the above inequality hold, namely, letting  $\delta>\delta_0$, where 
    \[\delta_0=\Big(\frac{1}{n_0}+1\Big)^{-1} \Big(\frac{1}{2}\Big)^\frac{1-m}{2p}.\]

For $n\geq1$, we need to show that 
    \begin{equation*}
        \begin{aligned}
            \frac{-e^{\lambda(n+n_0)}\Big(\frac{\delta}{n+n_0}-\frac{\delta}{n+n_0+1}\Big)^{m-1} + e^{\lambda(n+n_0-1)}\Big(\frac{\delta}{n+n_0-1}-\frac{\delta}{n+n_0}\Big)^{m-1}}{e^{\lambda(n+n_0)}+e^{\lambda(n+n_0-1)}} + \Big(\frac{\delta}{n+n_0}+\delta \Big)^p\\
             \times \Big[  \frac{e^{\lambda(n+n_0)}\Big(\frac{\delta}{n+n_0}-\frac{\delta}{n+n_0+1}\Big)^2 + e^{\lambda(n+n_0-1)}\Big(\frac{\delta}{n+n_0-1}-\frac{\delta}{n+n_0}\Big)^2}{e^{\lambda(n+n_0)}+e^{\lambda(n+n_0-1)}}\Big]^\frac{m-1}{2} \leqslant0,
        \end{aligned}
    \end{equation*}
    equivalently
    \begin{equation}\label{eq3.25}
        \delta^p\leqslant\Lambda_4(n+n_0),
    \end{equation}
    where
    \begin{equation*}
        \begin{aligned}
            \Lambda_4(n)\coloneqq 
            \frac{e^{\lambda n} \big(\frac{1}{n}-\frac{1}{n+1}\big)^{m-1} - e^{\lambda (n-1)} \big(\frac{1}{n-1}-\frac{1}{n}\big)^{m-1}}{e^{\lambda n}+e^{\lambda (n-1)}} \times \big(\frac{1}{n}+1\big)^{-p} \times 2^\frac{m-1}{2} \\
            \times \Big( \frac{e^{\lambda n} \big(\frac{1}{n}-\frac{1}{n+1}\big)^2 - e^{\lambda (n-1)} \big(\frac{1}{n-1}-\frac{1}{n}\big)^2}{e^{\lambda n}+e^{\lambda (n-1)}} \Big) ^{-\frac{m-1}{2}} .
        \end{aligned}
    \end{equation*}
    Since \[ \lim_{n\rightarrow\infty} \Lambda_4(n) = 2^\frac{m-1}{2}\frac{1-e^{-\lambda}}{1+e^{-\lambda}}, \]
    It follows that there exists some large $n_0$ such that for all $n\geqslant0$, the RHS of \eqref{eq3.25} is bounded from below by $\frac{2^\frac{m-1}{2}\frac{1-e^{-\lambda}}{1+e^{-\lambda}}}{2}$. Taking $\delta=\min \{\delta_0,\delta_1\}$,  where $\delta_1=\Big[\frac{2^\frac{m-1}{2}\frac{1-e^{-\lambda}}{1+e^{-\lambda}}}{2} \Big]^\frac{1}{p}$, \eqref{eq3.25} is true. The proof is completed.
\end{proof}

\begin{proof}[Proof of Theorem \ref{Thm1.3} (V)]
    Consider the following two cases:
\begin{enumerate}
\item[(V-1).]{$(p,q)\in G_{5.1};$}
\item[(V-2).]{$(p,q)\in G_{5.2}$ with $1<m<3.$}
\end{enumerate}
    In the case (V-1), we take $\mu$ and $u$ as
    \begin{equation}\label{eq3.26}
        \mu_{xy}=\mu_n=\frac{\lambda e^{\lambda n}}{(N-1)^n}, \text{ for any }(x,y)\in E_n,\text{ } n\geqslant0,
    \end{equation}

    \begin{equation}\label{eq3.27}
        u(x) = u_n =e^{-a\lambda n},\text{ for any }x\in D_n,\text{ } n\geqslant0,
    \end{equation}
    where $a=\min\{\frac{1}{4},\frac{1}{m}\}$ and $\lambda$ is to be determined. Under the above choice of $\mu$, we obtain that \eqref{eq1.13} holds. We next prove that $u$ is a nontrivial positive solution of \eqref{eq1}.

  For $n=0$, we need to show that
    \begin{equation}\label{eq3.28}
        1\leqslant 2^\frac{q}{2} (1-e^{-a\lambda})^{m-1-q}.
    \end{equation}
Since $q>0$ and $m-1\geq q$, we choose $\lambda$ large enough such that \eqref{eq3.28} is true.

For $n\geq 1$, we need to prove that
    \begin{equation}\label{eq3.29}
        1\leqslant 2^\frac{q}{2} (1-e^{-a\lambda})^{m-1-q} \bigg( \frac{1+e^{-\lambda}}{1+e^{(2a-1)\lambda}} \bigg)^\frac{q}{2} \bigg(\frac{1-e^{((m-1)a-1)\lambda}}{1+e^{-\lambda}} \bigg).
    \end{equation}
    Since $a=\min\{\frac{1}{4},\frac{1}{m}\}$, we have $2a-1<0$ and $(m-1)a-1<0$. Thus, \eqref{eq3.29} is also true when $\lambda$ is  large enough.

    In case (V-2), fix a vertex $o\in T_N$ and choose a special vertex $p\sim o$. Define 
    \[ P \coloneqq \{ x\in T_N| \text{ }o \text{ is not in the path between }x \text{ and }p\},\]
    \[ D'_{-n} \coloneqq \{ x\in P|\text{ }d(o,x)=n\} \text{ for any } n\geqslant0,\]
    \[D'_n\coloneqq\{x\in  T_N\setminus P|\text{ }d(o,x)=n\} \text{ for any } n\geqslant0 .\]
  Denote $E'_n$ as the set of all the edges from vertices in $D'_n$ to vertices in $D'_{n+1}$ for $n\geqslant0$.

    In this case, we take $\mu$ and $u$ as 
    \begin{equation}\label{eq3.30}
        \mu_{xy}=\mu_n=\frac{\lambda e^{\lambda n}}{(N-1)^n} \text{ for any }(x,y)\in E'_n, \text{ }n\geqslant0,
    \end{equation}

    \begin{equation}\label{eq3.31}
        \mu_{xy}=\mu_n=\frac{\lambda e^{\lambda n}}{(N-1)^{-n-1}} \text{ for any }(x,y)\in E'_n, \text{ }n\leqslant-1,
    \end{equation}

    \begin{equation}\label{eq3.32}
        u(x)=u_n=e^{-a(\lambda-1)n}\text{ for any }(x,y)\in D'_n,\text{ }n\in \mathbb Z,
    \end{equation}
    where $\frac{1}{2} <a< \frac{1}{m-1}$. Remark that this choice of $a$ is possible because $m-1<2$. Noting that $D_n=D'_{-n}\cup D'_n$, we have
    \[ W_o(n)=\sum_{k=0}^{n} \mu(D_k) \asymp \sum_{k=0}^{n}(N-1)^k (\mu_k+\mu_{-k}) \asymp e^{\lambda n},\] where $\lambda$ is to be determined.

    Next, we need to prove that, with the chosen $\mu$ and $u$, \eqref{eq1} holds. This is equivalent to verify that the following inequalities hold:
    
    \begin{equation}
    	\begin{aligned}\label{eq3.33}
    		&\frac{-(N-1)\mu_n(u_n-u_{n+1})^{m-1} + \mu_{n-1}(u_{n-1}-u_n)^{m-1}}{(N-1)\mu_n+\mu_{n-1}} \\
    		&\qquad+ u_n^p \bigg[\frac{(N-1)\mu_n(u_n-u_{n+1})^2 + \mu_{n-1}(u_{n-1}-u_n)^2}{2(N-1)\mu_n+2\mu_{n-1}}\bigg]^\frac{q}{2} \leqslant 0,
    	\end{aligned}
    \end{equation}
    
   for any $n\geqslant0$, and 
    
    \begin{equation}
    	\begin{aligned}\label{eq3.34}
    		&\frac{-\mu_n (u_n-u_{n+1})^{m-1} + (N-1)\mu_{n-1}(u_{n-1}-u_n)^{m-1}}{\mu_n+(N-1)\mu_{n-1}} \\
    		&\qquad+ u_n^p \bigg[\frac{\mu_n(u_n-u_{n+1})^2 + (N-1)\mu_{n-1}(u_{n-1}-u_n)^2}{2\mu_n+2(N-1)\mu_{n-1}}\bigg]^\frac{q}{2} \leqslant 0.
    	\end{aligned}
    \end{equation}
for any $n\leqslant-1$.

    From \eqref{eq3.30}-\eqref{eq3.32}, we deduce that \eqref{eq3.33} and \eqref{eq3.34} are equivalent to
    \begin{equation*}
        \begin{aligned}
            &\frac{-e^{\lambda n}\big(e^{-a(\lambda-1)n}-e^{-a(\lambda-1)(n+1) }\big)^{m-1} + e^{\lambda (n-1)}\big(e^{-a(\lambda-1)(n-1)}-e^{-a(\lambda-1)n }\big)^{m-1}}{e^{\lambda n}+ e^{\lambda (n-1)}} \\
            &+ e^{-a(\lambda-1)np }\times \\
            &\bigg[\frac{e^{\lambda n}\big(e^{-a(\lambda-1)n}-e^{-a(\lambda-1)(n+1) }\big)^2 + e^{\lambda (n-1)}\big(e^{-a(\lambda-1)(n-1)}-e^{-a(\lambda-1)n }\big)^2}{2e^{\lambda n}+ 2e^{\lambda (n-1)}} \bigg]^\frac{q}{2} \leqslant0,
        \end{aligned}
    \end{equation*}
    for any $n\in \mathbb{Z}$. That is
  \begin{equation}\label{eq3.35}
        1\leqslant 2^\frac{q}{2} (1-e^{-a(\lambda-1)})^{m-1-q} \bigg( \frac{1+e^{-\lambda}}{1+e^{(2a-1)\lambda-2a}} \bigg)^\frac{q}{2} \bigg(\frac{1-e^{((m-1)a-1)\lambda-a(m-1)}}{1+e^{-\lambda}} \bigg).
    \end{equation}
   It follows from $q<0$ that the RHS of \eqref{eq3.35} tends to infinity as $\lambda\rightarrow\infty$. Therefore, $u$ is a solution to \eqref{eq1} by choosing some large $\lambda$.

\end{proof}

\noindent\textbf{Author Contributions} The authors contributed equally to this work, and all authors reviewed the manuscript.

\noindent{\bf Acknowledgments}

A.T.D and D.T.Q  would like to thank Vietnam Institute for Advanced Study in Mathematics (VIASM) for hospitality.

\noindent\textbf{Data Availability} No datasets were generated or analyzed during the current study.

%\noindent\textbf{Funding}  Y. Liu was supported by China Scholarship Council (Grant No.202406200019). Y. Sun was supported by the National Natural 
% Science Foundation of China (Grant No.12371206).
\vskip 2ex
\noindent\textbf{Declarations}
\vskip 1ex
\noindent\textbf{Conflict of interest} The authors declare no conflict of interest.

\noindent\textbf{Ethical Approval} Not applicable
%    Text of artcle.  

%\bibliographystyle{amsplain}
	\bibliographystyle{abbrv}
	\bibliography{weightedgraph,henon-ref2}

\begin{thebibliography}{10}

\bibitem{AT20}
D.~Andreucci and A.~F. Tedeev.
\newblock Asymptotic estimates for the {$p$}-{L}aplacian on infinite graphs
  with decaying initial data.
\newblock {\em Potential Anal.}, 53(2):677--699, 2020.

\bibitem{Yau15}
F.~Bauer, P.~Horn, Y.~Lin, G.~Lippner, D.~Mangoubi, and S.-T. Yau.
\newblock Li-{Y}au inequality on graphs.
\newblock {\em J. Differential Geom.}, 99(3):359--405, 2015.

\bibitem{CHZ22}
C.~Chang, B.~Hu, and Z.~Zhang.
\newblock Liouville-type theorems and existence of solutions for quasilinear
  elliptic equations with nonlinear gradient terms.
\newblock {\em Nonlinear Anal.}, 220:Paper No. 112873, 29, 2022.

\bibitem{Filip09}
R.~Filippucci.
\newblock Nonexistence of positive weak solutions of elliptic inequalities.
\newblock {\em Nonlinear Anal.}, 70(8):2903--2916, 2009.

\bibitem{FP10}
R.~Filippucci, P.~Pucci, and M.~Rigoli.
\newblock Nonlinear weighted {$p$}-{L}aplacian elliptic inequalities with
  gradient terms.
\newblock {\em Commun. Contemp. Math.}, 12(3):501--535, 2010.

\bibitem{Ge18}
H.~Ge.
\newblock A {$p$}-th {Y}amabe equation on graph.
\newblock {\em Proc. Amer. Math. Soc.}, 146(5):2219--2224, 2018.

\bibitem{Ge20}
H.~Ge.
\newblock The {$p$}th {K}azdan-{W}arner equation on graphs.
\newblock {\em Commun. Contemp. Math.}, 22(6):1950052, 17, 2020.

\bibitem{GW25}
Y.~Ge and L.~Wang.
\newblock {$p $}-{L}aplace elliptic inequalities on the graph.
\newblock {\em Commun. Pure Appl. Anal.}, 24(3):389--411, 2025.

\bibitem{Gri85}
A.~Grigor'yan.
\newblock The existence of positive fundamental solutions of the {L}aplace
  equation on {R}iemannian manifolds.
\newblock {\em Mat. Sb. (N.S.)}, 128(170)(3):354--363, 446, 1985.

\bibitem{Gr18}
A.~Grigor'yan.
\newblock {\em Introduction to analysis on graphs}, volume~71 of {\em
  University Lecture Series}.
\newblock American Mathematical Society, Providence, RI, 2018.

\bibitem{GLY16b}
A.~Grigor'yan, Y.~Lin, and Y.~Yang.
\newblock Kazdan-{W}arner equation on graph.
\newblock {\em Calc. Var. Partial Differential Equations}, 55(4):Art. 92, 13,
  2016.

\bibitem{GLY17}
A.~Grigor'yan, Y.~Lin, and Y.~Yang.
\newblock Existence of positive solutions to some nonlinear equations on
  locally finite graphs.
\newblock {\em Sci. China Math.}, 60(7):1311--1324, 2017.

\bibitem{GS14}
A.~Grigor'yan and Y.~Sun.
\newblock On nonnegative solutions of the inequality {$\Delta u+u^\sigma\leq0$}
  on {R}iemannian manifolds.
\newblock {\em Comm. Pure Appl. Math.}, 67(8):1336--1352, 2014.

\bibitem{GSV20}
A.~Grigor'yan, Y.~Sun, and I.~Verbitsky.
\newblock Superlinear elliptic inequalities on manifolds.
\newblock {\em J. Funct. Anal.}, 278(9):108444, 34, 2020.

\bibitem{GHS23}
Q.~Gu, X.~Huang, and Y.~Sun.
\newblock Semi-linear elliptic inequalities on weighted graphs.
\newblock {\em Calc. Var. Partial Differential Equations}, 62(2):Paper No. 42,
  14, 2023.

\bibitem{HS21}
X.~l. Han and M.~Q. Shao.
\newblock {$p$}-{L}aplacian equations on locally finite graphs.
\newblock {\em Acta Math. Sin. (Engl. Ser.)}, 37(11):1645--1678, 2021.

\bibitem{HS23}
L.~Hao and Y.~Sun.
\newblock Sharp {L}iouville type results for semilinear elliptic inequalities
  involving gradient terms on weighted graphs.
\newblock {\em Discrete Contin. Dyn. Syst. Ser. S}, 16(6):1484--1516, 2023.

\bibitem{Hol99}
I.~Holopainen.
\newblock Volume growth, {G}reen's functions, and parabolicity of ends.
\newblock {\em Duke Math. J.}, 97(2):319--346, 1999.

\bibitem{HL21}
B.~Hua and R.~Li.
\newblock The existence of extremal functions for discrete {S}obolev
  inequalities on lattice graphs.
\newblock {\em J. Differential Equations}, 305:224--241, 2021.

\bibitem{HLW23}
B.~Hua, R.~Li, and L.~Wang.
\newblock A class of semilinear elliptic equations on groups of polynomial
  growth.
\newblock {\em J. Differential Equations}, 363:327--349, 2023.

\bibitem{HM15}
B.~Hua and D.~Mugnolo.
\newblock Time regularity and long-time behavior of parabolic {$p$}-{L}aplace
  equations on infinite graphs.
\newblock {\em J. Differential Equations}, 259(11):6162--6190, 2015.

\bibitem{HLY20}
A.~Huang, Y.~Lin, and S.-T. Yau.
\newblock Existence of solutions to mean field equations on graphs.
\newblock {\em Comm. Math. Phys.}, 377(1):613--621, 2020.

\bibitem{Karp82}
L.~Karp.
\newblock Subharmonic functions, harmonic mappings and isometric immersions.
\newblock In {\em Seminar on {D}ifferential {G}eometry}, volume No. 102 of {\em
  Ann. of Math. Stud.}, pages 133--142. Princeton Univ. Press, Princeton, NJ,
  1982.

\bibitem{LW17}
Y.~Lin and Y.~Wu.
\newblock The existence and nonexistence of global solutions for a semilinear
  heat equation on graphs.
\newblock {\em Calc. Var. Partial Differential Equations}, 56(4):Paper No. 102,
  22, 2017.

\bibitem{LW18}
Y.~Lin and Y.~Wu.
\newblock Blow-up problems for nonlinear parabolic equations on locally finite
  graphs.
\newblock {\em Acta Math. Sci. Ser. B (Engl. Ed.)}, 38(3):843--856, 2018.

\bibitem{LY20}
S.~Liu and Y.~Yang.
\newblock Multiple solutions of {K}azdan-{W}arner equation on graphs in the
  negative case.
\newblock {\em Calc. Var. Partial Differential Equations}, 59(5):Paper No. 164,
  15, 2020.

\bibitem{Liu23}
Y.~Liu.
\newblock Nonexistence of global solutions for a class of nonlinear parabolic
  equations on graphs.
\newblock {\em Bull. Malays. Math. Sci. Soc.}, 46(6):Paper No. 189, 22, 2023.

\bibitem{Liu24}
Y.~Liu.
\newblock Existence and {N}onexistence of {G}lobal {S}olutions to the
  {P}arabolic {E}quations on {L}ocally {F}inite {G}raphs.
\newblock {\em Results Math.}, 79(4):Paper No. 164, 2024.

\bibitem{MDQ25}
N.~C. Minh, A.~T. Duong, and D.~T. Quyet.
\newblock Liouville-type theorems for systems of elliptic inequalities
  involving $p$-laplace operator on weighted graphs.
\newblock {\em Commun. Pure Appl. Anal.}, 24(4):641--660, 2025.

\bibitem{MP01}
E.~Mitidieri and S.~I. Pohozaev.
\newblock A priori estimates and the absence of solutions of nonlinear partial
  differential equations and inequalities.
\newblock {\em Tr. Mat. Inst. Steklova}, 234:1--384, 2001.

\bibitem{MPS23}
D.~D. Monticelli, F.~Punzo, and J.~Somaglia.
\newblock Nonexistence results for semilinear elliptic equations on weighted
  graphs.
\newblock {\em Math. Ann.}, 2025.

\bibitem{MPS24}
D.~D. Monticelli, F.~Punzo, and J.~Somaglia.
\newblock Nonexistence of solutions to parabolic problems with a potential on
  weighted graphs.
\newblock {\em J. Differential Equations}, 453:Paper No. 113782, 24, 2026.

\bibitem{PZ26-1}
F.~Punzo and E.~Zucchero.
\newblock On a semilinear heat equation on infinite graphs i: blow-up for large
  initial data, 2026.

\bibitem{PZ26-2}
F.~Punzo and E.~Zucchero.
\newblock On a semilinear heat equation on infinite graphs ii: blow-up for
  arbitrary initial data and global existence, 2026.

\bibitem{Saloff-Coste95}
L.~Saloff-Coste.
\newblock Inequalities for {$p$}-superharmonic functions on networks.
\newblock {\em Rend. Sem. Mat. Fis. Milano}, 65:139--158, 1995.

\bibitem{Saloff-Coste97}
L.~Saloff-Coste.
\newblock Some inequalities for superharmonic functions on graphs.
\newblock {\em Potential Anal.}, 6(2):163--181, 1997.

\bibitem{SXX22}
Y.~Sun, J.~Xiao, and F.~Xu.
\newblock A sharp {L}iouville principle for {$\Delta_m u+u^p|\nabla u|^q\le0$}
  on geodesically complete noncompact {R}iemannian manifolds.
\newblock {\em Math. Ann.}, 384(3-4):1309--1341, 2022.

\bibitem{Varo83}
N.~T. Varopoulos.
\newblock Potential theory and diffusion on {R}iemannian manifolds.
\newblock In {\em Conference on harmonic analysis in honor of {A}ntoni
  {Z}ygmund, {V}ol. {I}, {II} ({C}hicago, {I}ll., 1981)}, Wadsworth Math. Ser.,
  pages 821--837. Wadsworth, Belmont, CA, 1983.

\bibitem{Woess00}
W.~Woess.
\newblock {\em Random walks on infinite graphs and groups}, volume 138 of {\em
  Cambridge Tracts in Mathematics}.
\newblock Cambridge University Press, Cambridge, 2000.

\bibitem{Wu21}
Y.~Wu.
\newblock Blow-up for a semilinear heat equation with {F}ujita's critical
  exponent on locally finite graphs.
\newblock {\em Rev. R. Acad. Cienc. Exactas F\'{\i}s. Nat. Ser. A Mat. RACSAM},
  115(3):Paper No. 133, 16, 2021.

\bibitem{Wu24}
Y.~Wu.
\newblock Blow-up conditions for a semilinear parabolic system on locally
  finite graphs.
\newblock {\em Acta Math. Sci. Ser. B (Engl. Ed.)}, 44(2):609--631, 2024.

\bibitem{ZL18}
X.~Zhang and A.~Lin.
\newblock Positive solutions of {$p$}-th {Y}amabe type equations on graphs.
\newblock {\em Front. Math. China}, 13(6):1501--1514, 2018.

\bibitem{ZL19}
X.~Zhang and A.~Lin.
\newblock Positive solutions of {$p$}-th {Y}amabe type equations on infinite
  graphs.
\newblock {\em Proc. Amer. Math. Soc.}, 147(4):1421--1427, 2019.

\end{thebibliography}
\end{document}